\documentclass[numbers,preprint]{elsarticle}
\usepackage{amsmath,amssymb,bbm,mathtools}
\usepackage{amsthm}
\usepackage{empheq}
\usepackage{eucal}
\usepackage{mathrsfs}
\usepackage{nicematrix}
\usepackage{enumitem}
\usepackage{booktabs}
\usepackage{multirow,multicol}
\usepackage{float}
\usepackage{algorithm,algorithmicx,algpseudocode}
\algdef{SE}[DOWHILE]{Do}{doWhile}{\algorithmicdo}[1]{\algorithmicwhile\ #1}%
\usepackage{array}
\usepackage{eqparbox}

\usepackage{makecell}
\usepackage{tabularx}
\usepackage{array}
\usepackage{kbordermatrix}
\usepackage{nicematrix}
\usepackage{stfloats}
\newcommand{\blue}[1]{{#1}}
\newcommand{\orange}[1]{{\color{orange}#1}}
\newcommand{\edit}[1]{{\color{blue}#1}}
\usepackage{orcidlink}
\usepackage{graphicx}
\usepackage{ragged2e}     
\usepackage{adjustbox}    
\newcolumntype{Y}{>{\RaggedRight\arraybackslash}X}

\usepackage{cleveref}

\usepackage{subcaption}

\usepackage{balance}

\usepackage{pgfplots}
\pgfplotsset{compat = 1.3}
\pgfplotscreateplotcyclelist{mycolorlist}{%
	blue,every mark/.append style={fill=blue!80!black},mark=*\\%
	red,every mark/.append style={fill=red!80!black},mark=square*\\%
	brown!60!black,every mark/.append style={fill=brown!80!black},mark=otimes*\\%
	black,mark=star\\%
	blue,every mark/.append style={fill=blue!80!black},mark=diamond*}

\newcommand*\widefbox[1]{\fbox{\hspace{2em}#1\hspace{2em}}}
\newcommand{\mcal}[1]{\mathcal{#1}}

\newcommand{\ldef}{\stackrel{\Delta}{=}}

\newcommand{\mb}[1]{\mathbf{#1}}

\newcommand{\tth}[1]{{#1}^{\text{th}}}

\newcommand{\bm}[1]{
	\begin{bmatrix}
		#1
	\end{bmatrix}
}

\newcommand{\real}[1]{\mathbb{R}^{#1}}

\newcommand{\nn}{\texttt{MoGE}}

\newcommand{\mleq}{\preceq}

\newtheorem{thm}{Theorem}
\newtheorem{lemma}{Lemma}
\newtheorem{prop}{Proposition}
\newtheorem{corollary}{Corollary}
\newtheorem{defn}{Definition}

\newtheorem{remark}{Remark}

\newtheorem*{claim*}{Claim}

\newtheorem{assumption}{Assumption}
\newtheorem*{thm*}{Theorem}

\newcommand{\tL}{\tilde{L}}

\newcommand{\yff}{y_{ij}^\text{ff}}
\newcommand{\yft}{y_{ij}^\text{ft}}
\newcommand{\ytt}{y_{ij}^\text{tt}}
\newcommand{\ytf}{y_{ij}^\text{tf}}
\newcommand{\pf}{P}
\newcommand{\qf}{Q}
\newcommand{\pt}{P}
\newcommand{\qt}{Q}
\newcommand{\pff}{\mathrm{pf}}
\newcommand{\re}[1]{\text{Re}\left(#1\right)}
\newcommand{\im}[1]{\text{Im}\left(#1\right)}

\allowdisplaybreaks[4]
\usepackage{verbatimbox}

\renewcommand{\kbldelim}{[}
\renewcommand{\kbrdelim}{]}
\newcommand{\soc}{\text{soc}}
\newcommand{\ang}{\text{ang}}
\newcommand{\moge}{$\texttt{MoGE}$}
\newcommand{\icnn}{\texttt{ICNN}}
\newcommand{\mgn}{\texttt{MGN}}
\newcommand{\gate}{\texttt{Gate}}

\newcommand{\ro}[1]{%
  \xrightarrow{\mathmakebox[\rowidth]{#1}}%
}
\newlength{\rowidth}
\AtBeginDocument{\setlength{\rowidth}{4em}}

\usetikzlibrary{positioning}
\usepackage[numbers]{natbib}

\newcolumntype{Y}{>{\raggedright\arraybackslash}X}

\renewcommand{\edit}[1]{{#1}}
\renewcommand{\orange}[1]{{#1}}

\begin{document}
	\begin{frontmatter}
		\title{\blue{Presolving Convexified Optimal Power Flow with Mixtures of Gradient Experts}}
		\tnotetext[1]{This work was partially supported by the 2023 CITRIS Interdisciplinary Innovation Program (I2P).}
		\author[1]{Shourya Bose\,\orcidlink{0000-0002-2081-9545}}
		\ead{shbose@ucsc.edu}
		\author[1]{Kejun Chen\,\orcidlink{0000-0002-2162-7887}}
		\ead{kchen158@ucsc.edu}
		\author[1]{Yu Zhang\,\orcidlink{0000-0001-7889-2676}}
		\ead{zhangy@ucsc.edu}
		\affiliation[1]{organization={University of California Santa Cruz}, addressline={1156 High St.}, city={Santa Cruz}, postcode={CA 95064}, country={USA}}	
		
		\begin{abstract}
			\edit{
				Convex relaxations and approximations of the optimal power flow (OPF) problem have gained significant research and industrial interest for planning and operations in electric power networks. One approach for reducing their solve times is \emph{presolving} which eliminates constraints from the problem definition, thereby reducing the burden of the underlying optimization algorithm.
				To this end, we propose a presolving framework for convexified optimal power flow (C-OPF) problems, which uses a novel deep learning-based architecture called \moge\ (Mixture of Gradient Experts). In this framework, problem size is reduced by learning the mapping between C-OPF parameters and optimal dual variables (the latter being representable as gradients), which is then used to screen constraints that are non-binding at optimum. The validity of using this presolve framework across arbitrary families of C-OPF problems is theoretically demonstrated. We characterize generalization in \moge\ and develop a post-solve recovery procedure to mitigate possible constraint classification errors. Using two different C-OPF models, we show via simulations that our framework reduces solve times by upto 34\% across multiple PGLIB and MATPOWER test cases, while providing an identical solution as the full problem.
			}
		\end{abstract}
	\end{frontmatter}

	\section{Introduction}
	The Optimal Power Flow (OPF) determines the optimal settings for control variables in an electric power network. Its primary goal is to minimize the total generation cost while satisfying various operational constraints (e.g., generation limits, voltage limits, and thermal limits). \blue{OPF is solved by different entities such as independent system operators~\cite{opf-history-ferc}, energy markets~\cite{ZQ-etal:2009-energymarkets}, and microgrid operators~\cite{SB-microgrids} as a part of their planning and operational procedures.}
	{Increasing stochastic factors such as load demand fluctuations and varying generation from renewable distributed energy resources (DERs) necessitate frequent re-solving of OPF, which in turn calls for time-efficient solution methods.\par
		Convexified formulations of AC optimal power flow (ACOPF), which we denote by C-OPFs, have recently gained significant research and industrial attention. C-OPF formulations benefit from theoretical guarantees of global optimality and the availability of mature commercial solvers. Popular C-OPF formulations include DC optimal power flow (DCOPF)~\cite{AJW-BFW-GBS:2013}, convexified DistFlow OPF~\cite{RAJ-2006,MF-CRC-SHL-KMC:2011,MF-SHL:2013-1}, semidefinite relaxation (SDR)~\cite{XB-HW-KF-YW:2008,RM-SS-JL:2015,JL-SHL:2012}, second-order cone programming (SOCP) relaxation~\cite{SB-SHL-TT-BH:2015}, quadratic-convex relaxation~\cite{CC-HLH-PVH:2016}, McCormick relaxation~\cite{BK-SSD-XAS:2016}, etc. A comprehensive review of convex relaxations and approximations of OPF can be found in~\cite{DKM-IAH:2019}. }
	
	\subsection{Motivation and Scope}
	In this paper, we aim to reduce the solution time of C-OPF solvers using methods that are agnostic to the underlying optimization algorithm. This is achieved by using the presolving approach of \emph{Constraint Screening (CS)}, which identifies and eliminates constraints that are non-binding at optimum, thereby reducing problem size and accelerating solution times. \blue{This is done by learning the mapping between problem parameters and optimal dual variables using our proposed \moge\ architecture, which allows identifying binding constraints using the Kahrush-Kuhn-Tucker (KKT) conditions}. These problem parameters include load demands and possibly other quantities which vary between different C-OPF instances. 
	
    We restrict our attention to C-OPFs rather than the nonconvex ACOPF due to several reasons. C-OPFs are widely used as a faithful approximation of ACOPF~\cite{ZY-etal:2017}, with the workhorse DCOPF being a linear, and therefore convex approximation of ACOPF. Moreover, SOCP- and SDP-based C-OPFs allow recovery of optimal ACOPF solutions under certain circumstances~\cite{low-exactness,lavaei-zero-duality-gap}. Lastly, convex relaxations of ACOPF provide a lower bound to the optimal ACOPF objective value, which is useful in branch-and-bound algorithms for solving mixed-integer formulations of ACOPF~\cite{SSS-JHT-1997}. \par
    \edit{
    We also highlight that CS differs fundamentally from many other methods to accelerate OPF. For example, methods based on machine learning~\cite{low-unsupervised,PLD-DR-JZC:2021} learn the mapping from input parameters to a partial set of OPF variables, with the remaining variables computed algorithmically. Similarly, our previous works~\cite{zhanglab-1,zhanglab-2} use physics-informed neural networks and gradient learning to directly learn OPF solutions, or parts thereof. Some other methods for accelerating OPF include graph neural networks~\cite{gnn-opf}, evolutionary algorithms~\cite{evo1,evo2}, simulated annealing~\cite{sa-opf}, and particle swarm optimization~\cite{pso-opf}. The aforementioned methods, while benefiting from fast solutions, lack provable guarantees of local optimality or even feasibility of the resulting solutions due to generalization errors inherent to machine learning techniques. On the other hand, the present work identifies CS as an approach which can result in provably feasible and optimal solutions, even with the use of machine learning.
    }

	\subsection{Prior Work}
	Earliest works on the removal of redundant and non-binding constraints are for linear programs~\cite{JT-1983,GLT-FMT-SZ:1966}. Recent domain-specific CS methods for power systems are based on geometry of the feasible set or machine learning techniques.
	CS methods based on feasible set geometry analyze the constraints \emph{a priori} to generate a list of constraints which are provably non-binding at the optimum. In this category, various algorithms have been proposed to eliminate redundant line flow constraints for DCOPF formulations, based on solving auxiliary optimization problems (see~\cite{ZJZ-etal:2020,BH-etal:2013,RW-RM:2020}). Similarly, different algorithms have been developed to accelerate the branch-and-bound algorithm for unit commitment by eliminating network constraints (see~\cite{SZ-etal:2021,AP-SP-JMM-JCA:2023}). \blue{However, geometry-based methods require expensive re-computations for every \emph{test case} (for example a 30-bus system vs. a 141-bus system), which involves solving auxiliary optimization problems whose count is proportional to the number of constraints being targeted for elimination. Furthermore, since these methods aim to remove non-binding constraints over \emph{all} possible instances, the actual number of constraints removed are fewer.}

	Machine learning-based CS methods leverage historical or synthetic OPF datasets to train learning models, with load demands typically included as key problem parameters. For example, Deka and Misra~\cite{DD-SM:2019} and Hasan and Kargarian~\cite{FH-AK:2022} train deep neural network (DNN) based classifiers to learn the mapping between problem parameters and the binding constraints directly. \blue{These works fall under the umbrella category of \emph{active-set methods}, which aim to learn the optimal set of binding constraints from $2^M$ possible configurations of a problem that contains $M$ candidate constraints.}
	Dual variables play a pivotal role in DNN applications for OPF, as their optimal values reveal the binding status of corresponding constraints~\cite{FF-TWKM-PVH:2020,MKS-VK-GBG-2022}. Chen, Zhang, Chen, and Zhang~\cite{LC-YZ-BZ-2022} learn optimal dual variables with DNNs to solve DCOPF via solving the dual problem.  Nellikkath and Chatzivasileiadis~\cite{RN-SC-2022} leverage the KKT system to improve the robustness of deep learning for learning optimal solutions.

	\subsection{Contribution}
	Based on learning the dual variables, we propose a CS method to reduce the size of C-OPF problems in this paper. We first \blue{present a generic representation which can denote any arbitrary C-OPF} as a parametric convex optimization problem, with parameters representing values such as load demands, thermal limits, etc. We develop a novel neural network architecture called \emph{\moge\ (Mixture of Gradient Experts)} to learn the mapping from parameters to optimal dual variables. Once trained, \moge\ predicts optimal dual variables for new C-OPF instances, enabling the identification of non-binding constraints via complementary slackness. Removing these non-binding constraints reduces the problem size, thereby accelerating the solution process. In addition, we address the issue of potentially misclassified constraints by incorporating a recovery process. Any identified violated constraints after solving the reduced C-OPF are retroactively added to the active constraints, and the problem is re-solved. We harness the convexity of C-OPF to guarantee that even with mispredicted constraints, \blue{a finite number of re-solves are required to achieve an exact solution}.
	
	Our proposed approach presents significant advantages over existing active-set methods such as~\cite{DD-SM:2019,FH-AK:2022}. We use dual variables as a proxy to identify binding constraints, which makes our method \emph{physics-informed}.
	 We also address weaknesses inherent in other dual variable-based methods such as~\cite{LC-YZ-BZ-2022,RN-SC-2022}. First, our framework can handle nonlinear convex constraints, necessitating an analysis of the existence and uniqueness of the dual solution, as well as the validity of using dual variables for constraint classification. We provide the necessary analyses to support this. Second, previous methods lack provable robustness against mispredictions of dual variables, instead only offering theoretical generalization guarantees which are difficult to quantify for practical C-OPF cases. We address this issue through our recovery procedure. Finally, we address practical issues with implementation such as dataset augmentation for efficient utilization of data, and perform extensive simulations on a variety of test cases from the PGLIB OPF~\cite{PGLIB} and MATPOWER~\cite{matpower} libraries to validate the merits of our proposed method.
	
	\subsection{Organization and Notation}
	The paper is structured as follows. Section~\ref{sec:2} introduces a standard form of C-OPF that encompasses various OPF models, with the quadratic-convex (QC) \edit{and convexified DistFlow} (CDF) relaxations of ACOPF serving as specific examples. Section~\ref{sec:3} begins with an exploration of the analytical properties of C-OPF. Subsequently, we establish strong duality results for QC-OPF and CDF-OPF, which guarantees uniqueness of dual solutions. \blue{This analysis is then generalized into a test which can be performed for arbitrary families of C-OPF}. This section also details the design and training of \moge, along with the data generation process \edit{and the importance of dense sampling for the same}. Section~\ref{sec:4} conducts simulations to comparatively analyze the proposed method against other CS approaches. The conclusion is presented in Section~\ref{sec:5}, accompanied by key proofs provided in the Appendix.
	
	\emph{Notation:} Matrices and vectors are depicted in bold typeface. 
	The $n$-dimensional real space is denoted by $\mathbb{R}^n$. $[n]$ denotes the set $\{1,\dots,n\}$. For a vector $\mb{a}\in\real{n}$, $\mb{a}_i$ is its $\tth{i}$ element while $[\mb{a}]_{\mcal{S}}$ denotes the subvector corresponding to a set $\mcal{S}\subseteq [n]$. $|\mcal{S}|$ is the cardinality of set $\mcal{S}$. $\mb{a}\mleq \mb{b}$ denotes elementwise inequality between the two vectors, while $(\mb{a},\mb{b})$ denotes their concatenation. $\nabla f(\mb{x})$ denotes the gradient or Jacobian matrix of the function $f(\mb{x})$ while $\nabla^2 f(\mb{x})$ is its Hessian. $\mb{e}_i$ is the $\tth{i}$ canonical vector and $\mb{I}_n$ is the $n\times n$ identity matrix.
	$\mathrm{conv}\{\mb{x}_1,\dots,\mb{x}_n\}\ldef \{\sum_{i=1}^n\pmb{\alpha}_i\mb{x}_i~|~\pmb{\alpha} \succeq \mb{0},\mb{1}^{\top}\pmb{\alpha}=1\}$ denotes the convex hull of a finite set $\{\mb{x}_i\}_{i=1}^n$.
	$\mathsf{j} = \sqrt{-1}$ is the imaginary unit.

	%
	\section{C-OPF as Parametric Convex Optimization}
	\label{sec:2}
	In this section, we represent C-OPF as a parametric convex optimization problem. This allows us to generate a well-defined mapping between problem parameters and optimal dual variables, which can then be used for CS. In general, constraints for any OPF problem can broadly be categorized into two types: the power flow equations and other operational constraints. The power flow model remains fixed across problem instances, while operational constraints, such as nodal power balance and thermal limits, may vary across problem instances. These constraints' bounds serve as problem parameters in our formulation.
	
	\subsection{Generic C-OPF Formulation}
	Concretely, we represent any C-OPF as the following parametric optimization problem:
	\begin{subequations}
		\label{prob:orig}
		\begin{empheq}[box=\widefbox]{align}
			\label{eq:obj-orig}
			\mcal{V}(\pmb{\gamma},\pmb{\xi})\ldef\min\limits_{\mb{x}} \quad& f(\mb{x})\\
			\label{eq:non-param-cons-orig}
			\text{s.t.} \quad& \mb{g}(\mb{x})\mleq \mb{0}\;\;(\pmb{\lambda}),\quad \mb{h}(\mb{x}) = \mb{0}\;\;(\pmb{\mu})\\
			\label{eq:param-cons-orig}
			\quad& \tilde{\mb{g}}(\mb{x}) \mleq \pmb{\gamma}\;\;(\tilde{\pmb{\lambda}}), \quad \tilde{\mb{h}}(\mb{x}) = \pmb{\xi}\;\;(\tilde{\pmb{\mu}})\\
			\label{eq:var-bounds}
			\quad& {\underline{\mb{x}} \preceq \mb{x} \preceq \bar{\mb{x}}}\;\;(\underline{\pmb{\nu}},\bar{\pmb{\nu}}).
		\end{empheq}
	\end{subequations}
	Here, $\mb{x}\in\real{n}$ is the decision variable whose exact description depends on the underlying power flow model. The convex objective function $f:\real{n}\mapsto\real{}$ typically represents the total generation cost or power line losses over the network. Constraint functions $\mb{g}:\real{n}\mapsto \real{L}$ and $\mb{h}:\real{n}\mapsto \real{M}$ denote the power flow model, while $\tilde{\mb{g}}:\real{n}\mapsto \real{\tilde{L}}$ and $\tilde{\mb{h}}:\real{n}\mapsto \real{\tilde{M}}$ denote all other operational constraints. The dual variables for each block of constraints are represented alongside in parentheses. \blue{We use the terminology that $\mcal{V}$ forms a \emph{family} of C-OPF problems under a given power flow model, while its specific values for a given topology and different $(\pmb{\gamma},\pmb{\xi})$ form \emph{instances}.}
	
	The bounds for the operational constraints are captured by the right-hand side parameters $\pmb{\gamma}$ and $\pmb{\xi}$, which are used to parameterize the problem. 
	For the parameter space, we consider open, nonempty subsets $\Gamma\subset \real{\tilde{L}}$ and $\Xi\subset \real{\tilde{M}}$ such that problem~\eqref{prob:orig} is feasible for all $(\pmb{\gamma},\pmb{\xi})\in \Gamma \times \Xi$. 
	
	\begin{defn}
		The feasible power flow set is defined as
		\begin{align*}
			\mcal{X}_{\mathrm{pf}} \ldef \{ \mb{x}: \mb{g}(\mb{x})\mleq \mb{0},\mb{h}(\mb{x})=\mb{0} \}.
		\end{align*}
	\end{defn}
	\begin{defn}
		The active set of the inequality constraint $\mb{g}(\mb{x})\mleq \mb{0}$ at point $\mb{x}$ is defined as
		\begin{align*}
			\mcal{A}_{\mb{g}}(\mb{x}) \ldef \{ i\in[L]:\mb{g}_i(\mb{x})=0 \}.
		\end{align*}
	\end{defn}

	In the sequel, we show how the quadratic-convex relaxation to OPF (QC-OPF)~\cite{CC-HLH-PVH:2016} can be represented in the form of~\eqref{prob:orig}. 
	QC relaxation is generally tighter than SOCP relaxation and faster to compute than SDP relaxation~\cite{SB-SHL-TT-BH:2015,CC-HLH-PVH:2016}.
	While QC-OPF serves as our demonstrative example, the proposed framework and analysis are applicable to other C-OPF models. \edit{For example, the present framework can also be applied to the convexified DistFlow OPF model (CDF-OPF), which is a nonlinear representation of power flows in low-voltage radial power networks. We relegate said analysis to~\ref{app:cdf-opf}.}
	
    \begin{table*}[tb!]
        \centering
        \caption{Representing QC-OPF~\eqref{eq:qc-all} as parametric convex optimization problem~\eqref{prob:orig}, where the box constraint~\eqref{eq:var-bounds} corresponds to constraints~\eqref{eq:qc-gen-lims}, \eqref{eq:qc-volt-lims}, and \eqref{eq:qc-volt-cross-lims}.}
        \resizebox{\linewidth}{!}{
        \renewcommand{\arraystretch}{1.4}
\begin{tabular}{l p{4.5cm}|l p{2.25cm}|l p{4cm}}
	\hline
	\multicolumn{2}{c|}{\textbf{Variables}} & \multicolumn{2}{c|}{\textbf{Constraints}} & \multicolumn{2}{c}{\textbf{Parameters}}\\ \hline
	\multirow{3}{*}{$\mb{x}$:} & \multirow{3}{*}{\begin{minipage}{0.2\linewidth}
		$\{P^g_k,Q^g_k\}_{k\in\mcal{G}}$,\\$\{P_{ij},Q_{ij} \}_{(i,j)\in\mcal{E}\cup\mcal{E}^R}$,\\
			$\{W^{\text{R}}_{ij},W^{\text{I}}_{ij}\}_{(i,j)\in\mcal{E}},
   \{W_{ii}\}_{i\in\mcal{N}}$
	\end{minipage}} & $g(\mb{x}), \tilde{g}(\mb{x})$: & \eqref{eq:qc-soc}--\eqref{eq:qc-ang-lims}, \eqref{eq:qc-flowlims} &$\pmb{\gamma}$: & \multirow{2}{*}{\begin{minipage}{0.2\linewidth}
	$\{\bar{S}_{ij}^2 \}_{(i,j)\in\mcal{E}}, \{\bar{S}_{ij}^2 \}_{(i,j)\in\mcal{E}}$\\{}
	\end{minipage}}\\
	& & & & & \\
	&  &$h(\mb{x}), \tilde{h}(\mb{x})$: & \eqref{eq:qc-flows}, \eqref{eq:qc-balance} & $\pmb{\xi}$: & $\{P^d_i,Q^d_i\}_{i\in\mcal{N}}$\\
	\hline
\end{tabular}
\renewcommand{\arraystretch}{1}
        }
        \label{tab:convert_to_param}
    \end{table*}
	
	\subsection{QC-OPF} 
	QC-OPF is a convex relaxation to the ACOPF for meshed networks, possibly multiple generators on each bus, and transformers and phase shifters on each line. Let $(\mcal{N},\mcal{E})$ denote the directed graph representing the power network, and $\mcal{E}^R$ the edges with the reversed orientation. 
	Let $N=|\mcal{N}|$ and $E = |\mcal{E}|$ denote the total number of buses and power lines.
	Each bus $i\in\mcal{N}$ is associated with the indices $\mcal{G}_i$ of generators attached to it, real and reactive demands $P^d_i,Q^d_i$, line charging admittance $g^s_i-\mathsf{j}b^s_i$, and minimum (maximum) voltage magnitude $\underline{V}_i$ ($\bar{V}_i$). 
	In addition, the real and reactive power outputs of the $\tth{k}$ generator are denoted by $P^g_k$ and $Q^g_k$, respectively. Set $\mcal{G}\ldef \cup_{i\in\mcal{N}}\mcal{G}_i$ collects all generation units.
	For each branch $(i,j)\in\mcal{E}\cup \mcal{E}^R$ we have the thermal limit $\bar{S}_{ij}$, and the minimum and maximum phase angle differences (PAD) such that $-\frac{\pi}{2}<\underline{\theta}_{ij} < \bar{\theta}_{ij}<\frac{\pi}{2}$.
	Each branch is represented by a $\Pi$-model with the corresponding line matrix
	\begin{align*}
		\bm{ \yff & \yft\\ \ytf & \ytt} \ldef \bm{ \left( y_{ij} + \mathsf{j}\frac{b^c_{ij}}{2} \right)\frac{1}{|T_{ij}|^2}
			& -\frac{y_{ij}}{T_{ij}^*} \\ -\frac{y_{ij}}{T_{ij}} & y_{ij} + \mathsf{j}\frac{b_{ij}^c}{2}},
	\end{align*}
	where we have line admittance $y_{ij}$, total charging susceptance $b^c_{ij}$, and transformer parameter $T_{ij} = \tau_{ij}\angle \theta^{\text{shift}}_{ij}$. 
	
	Let $V_i$ denote the voltage phasor at bus $i$ and define $W_{ij}=W^{\text{R}}_{ij}+\mathsf{j}W^{\text{I}}_{ij}\ldef V_iV_j^*$. To this end, the QC-OPF formulation is given as follows.
	\begin{subequations}
		\label{eq:qc-all}
		\begin{align}
			\label{eq:qc-obj}
			&\min \, \sum_{k\in\mcal{G}} c_{k,2}\left(P^g_k\right)^2 + c_{k,1}P^g_k + c_{k,0} \\ 
			&\text{subject to:} \notag \\
			&\left.
			\begin{array}{l}
				\pf_{ij} \orange{-} \re{\yff}W_{ii} \orange{-} \re{\yft}W^{\text{R}}_{ij} \orange{-} \im{\yft}W^{\text{I}}_{ij} = 0\\
				\qf_{ij} \orange{+} \im{\yff}W_{ii} \orange{-} \re{\yft}W^{\text{I}}_{ij} \orange{+} \im{\yft}W^{\text{R}}_{ij} = 0\\
				\pt_{ji} \orange{-} \re{\ytt}W_{jj} \orange{-} \re{\ytf}W^{\text{R}}_{ij} \orange{+} \im{\ytf}W^{\text{I}}_{ij} = 0\\
				\qt_{ji} \orange{+} \im{\ytt}W_{jj} \orange{+} \re{\ytf}W^{\text{I}}_{ij} \orange{+} \im{\ytf}W^{\text{R}}_{ij} = 0
			\end{array}
			\right\} 
			\left. \, \forall (i,j)\in\mcal{E}\right. \label{eq:qc-flows} \\
			\label{eq:qc-soc}
			&(W^{\text{R}}_{ij})^2 + (W^{\text{I}}_{ij})^2 - W_{ii}W_{jj} \leq 0,\, \forall (i,j)\in\mcal{E} \\
			\label{eq:qc-ang-lims}
			&\tan\left( \underline{\theta}_{ij} \right)W^{\text{R}}_{ij}\leq W^{\text{I}}_{ij}
			\leq \tan \left( \bar{\theta}_{ij}\right)W^{\text{R}}_{ij},\, \forall (i,j)\in\mcal{E}\\
			\label{eq:qc-balance}
			&\left.
			\begin{array}{l}
				\left( \sum\limits_{k\in\mcal{G}_i} P^g_k\right) -g^s_iW_{ii} - \sum\limits_{(i,j)\in\mcal{E}\cup\mcal{E}^R}P_{ij} = P^d_i\\
				\left( \sum\limits_{k\in\mcal{G}_i} Q^g_k\right) +b^s_iW_{ii} - \sum\limits_{(i,j)\in\mcal{E}\cup\mcal{E}^R}Q_{ij} = Q^d_i
			\end{array}\right\},\, \forall i \in\mcal{N}\\
			\label{eq:qc-flowlims}
			&P_{ij}^2 + Q_{ij}^2 \leq \bar{S}_{ij}^2,\quad
			P_{ji}^2 + Q_{ji}^2 \leq \bar{S}_{ij}^2,\, \forall (i,j)\in\mcal{E}\\
			\label{eq:qc-gen-lims}
			&
			\underline{P}_k^g \leq P^g_k \leq \bar{P}^g_k,\quad
			\underline{Q}_k^g \leq Q^g_k \leq \bar{Q}^g_k,\,\forall k \in \mcal{G}\\
			\label{eq:qc-volt-lims}
			&\underline{V}_i^2 \leq W_{ii} \leq \bar{V}_i^2,\,\forall i \in \mcal{N}\\
			\label{eq:qc-volt-cross-lims}
			&
			|W^{\text{R}}_{ij}| \leq \bar{V}_i\bar{V}_j,\quad
			|W^{\text{I}}_{ij}| \leq \bar{V}_i\bar{V}_j,\,
			\forall(i,j)\in\mcal{E}.
		\end{align}
	\end{subequations}
	In the above, the objective~\eqref{eq:qc-obj} is a convex quadratic function of real power generation cost. \eqref{eq:qc-flows} describes the forward and reverse flows on each line, while~\eqref{eq:qc-soc} denotes the SOCP relaxation of the cross-term dependence. \eqref{eq:qc-ang-lims} enforces the PAD limits on each branch.~\eqref{eq:qc-balance} represents the nodal power balance, while line flow constraints~\eqref{eq:qc-flowlims} enforce the thermal limits. Finally,~\eqref{eq:qc-gen-lims},~\eqref{eq:qc-volt-lims}, and~\eqref{eq:qc-volt-cross-lims} represent the upper and lower bounds for the decision variables. 
	
	We can represent the QC-OPF~\eqref{eq:qc-all} in the generic formulation~\eqref{prob:orig}, as shown in Table~\ref{tab:convert_to_param}. The parameterized inequality $\tilde{\mb{g}}(\mb{x})\preceq \pmb{\gamma}$ contains the line flow constraints and the equality $\tilde{\mb{h}}(\mb{x})=\pmb{\xi}$ contains the power balance equations. The former by is parameterized by the thermal limits, and the latter by real and reactive load demands.

	\section{Constraint Screening for C-OPF}
	\label{sec:3}
	
	We divide this section into two parts. The first part provides analysis for generalized C-OPF of the form~\eqref{prob:orig}. In the second part, we present the proposed accelerated solution based on constraint screening.
	
	\subsection{Analytical Properties of C-OPF}
	We start with a standard result of the value function $\mcal{V}(\pmb{\gamma},\pmb{\xi})$.
	\begin{lemma}[Convexity of value function~{\cite[pp. 250]{SB-LV:2004}}]
		\label{lem:conv}
		$\mcal{V}(\pmb{\gamma},\pmb{\xi})$ is jointly convex in $\pmb{\gamma}$ and $\pmb{\xi}$ over any convex subset of $\Gamma\times \Xi$.
	\end{lemma}
	
	\begin{corollary}[Feasibility in the parameter space]
		\label{coro:convex_comb_params}
		Assume problem~\eqref{prob:orig} is feasible for each point in a finite set $\{(\pmb{\gamma}^{(k)},\pmb{\xi}^{(k)})\}_{k\in [K]}$. Then, the problem remains feasible for any point $({\pmb{\gamma}},{\pmb{\xi}}) \in \mathrm{conv}\{(\pmb{\gamma}^{(1)},\pmb{\xi}^{(1)}),\dots, (\pmb{\gamma}^{(K)},\pmb{\xi}^{(K)})\}$.
	\end{corollary}
	This corollary supports our data generation Algorithm~\ref{alg:dataset}, 
	and has a simple proof. $\mcal{V}(\pmb{\gamma},\pmb{\xi})$ is finite because the problem is feasible for each point $(\pmb{\gamma}^{(k)},\pmb{\xi}^{(k)})$. Due to the convexity of $\mcal{V}$, we have
	$\mcal{V}(\tilde{\pmb{\gamma}},\tilde{\pmb{\xi}})\leq\sum_k \pmb{\alpha}_k\mcal{V}(\pmb{\gamma}^{(k)},\pmb{\xi}^{(k)}) < \infty$, where $(\tilde{\pmb{\gamma}},\tilde{\pmb{\xi}})\ldef \sum_k \pmb{\alpha}_k(\pmb{\gamma}^{(k)},\pmb{\xi}^{(k)})$ for any $\pmb{\alpha} \succeq \mb{0}$ and $\mb{1}^\top \pmb{\alpha}=1$. This confirms that the problem is feasible for $(\tilde{\pmb{\gamma}},\tilde{\pmb{\xi}})$.
	%
	The {partial Lagrangian} function of~\eqref{prob:orig} can be written as
	\begin{align*}
		\mcal{L}(\mb{x},\pmb{\lambda},\tilde{\pmb{\lambda}},\pmb{\mu},\tilde{\pmb{\mu}}) \ldef f(\mb{x}) + \pmb{\lambda}^\top \mb{g}(\mb{x}) + \pmb{\mu}^\top \mb{h}(\mb{x}) 
		+ \tilde{\pmb{\lambda}}^\top (\tilde{\mb{g}}(\mb{x})-\pmb{\gamma}) + \tilde{\pmb{\mu}}^\top (\tilde{\mb{h}}(\mb{x})-\pmb{\xi}).
	\end{align*}
	The dual problem is
	\begin{align}
		\label{eq:dual_problem}
		\sup\limits_{\pmb{\lambda}\succeq\mb{0},\tilde{\pmb{\lambda}}\succeq\mb{0},\pmb{\mu},\tilde{\pmb{\mu}}} D(\pmb{\lambda},\tilde{\pmb{\lambda}},\pmb{\mu},\tilde{\pmb{\mu}}),
	\end{align}
	with the dual objective  
	\begin{align}
		\label{eq:dual_obj}
		D(\pmb{\lambda},\tilde{\pmb{\lambda}},\pmb{\mu},\tilde{\pmb{\mu}}) &\ldef \inf\limits_{{\underline{\mb{x}}\preceq\mb{x}\preceq\bar{\mb{x}}}} 	\mcal{L}(\mb{x},\pmb{\lambda},\tilde{\pmb{\lambda}},\pmb{\mu},\tilde{\pmb{\mu}}) \notag\\ 
		&= 
		-\tilde{\pmb{\lambda}}^\top \pmb{\gamma} - \tilde{\pmb{\mu}}^\top \pmb{\xi} + \inf\limits_{{\underline{\mb{x}}\preceq\mb{x}\preceq\bar{\mb{x}}}} \big\{ f(\mb{x}) + \pmb{\lambda}^\top \mb{g}(\mb{x}) \notag\\
		&\hspace{1cm}+\pmb{\mu}^\top \mb{h}(\mb{x}) + \tilde{\pmb{\lambda}}^\top \tilde{\mb{g}}(\mb{x}) + \tilde{\pmb{\mu}}^\top \tilde{\mb{h}}(\mb{x})\big\}.
	\end{align} 
	If strong duality holds (meaning the optimal values of the primal and dual problems are equal),
	the structure in~\eqref{eq:dual_obj} indicates that gradients of $\mcal{V}(\pmb{\gamma},\pmb{\xi})$ exist and can be expressed as functions of the optimal dual solutions.
	Strong duality is not guaranteed in general, though it always holds for feasible linear programs~\cite[Chap.~6]{JM-BG:2006}.
	For convex problems, various results establish conditions under which strong duality holds. These conditions are called constraint qualifications (a.k.a. regularity conditions). Constraint qualifications guarantee that a constrained minimizer satisfies the well-defined KKT system.
	
	Next, we show that the special structure of QC-OPF~\eqref{eq:qc-all} ensures strong duality by satisfying the \emph{linear independence constraint qualification (LICQ)}. We then generalize this finding to a test which can be used on arbitrary C-OPF models.
	\begin{lemma}[Fixed LICQ]
		\label{lemma:fixLICQ}
		Given QC-OPF~\eqref{eq:qc-all}, for all $\tilde{\mb{x}}\in\mcal{X}_{\pff}$ the LICQ holds with respect to the ``fixed'' constraints $\mb{g}$ and $\mb{h}$, i.e., the Jacobian matrix of the binding constraints is full row rank\footnote{For simplicity, full row rank is referred to as full rank hereinafter.}: 
		\begin{align*}
			\mathrm{rank} \left(\mb{J}(\tilde{\mb{x}})\ldef \bm{ \nabla [\mb{g}(\tilde{\mb{x}})]_{\mcal{A}_{\mb{g}}(\tilde{\mb{x}})}\\ \nabla \mb{h}(\tilde{\mb{x}})} \right) = M + |\mcal{A}_{\mb{g}}(\tilde{\mb{x}})|.
		\end{align*}
	\end{lemma}
	The proof is provided in~\ref{proof:fixLICQ}. This lemma allows us to extend LICQ to all constraints of QC-OPF over the entire parameter space. This is established in the following theorem.
	
	\begin{thm}[Genericity of LICQ]
		\label{thm:LICQ}
		For every feasible point of QC-OPF~\eqref{eq:qc-all} and for 
		almost everywhere in the parameter space $\Gamma \times \Xi$, the LICQ holds with respect to all constraints $\mb{g},\mb{h},\tilde{\mb{g}},\tilde{\mb{h}}$ upon satisfaction of Lemma~\ref{lemma:fixLICQ}. Thus, strong duality holds for QC-OPF.
	\end{thm}
	The proof is provided in~\ref{app:proof-thm-LICQ}. For the convex QC-OPF, strong duality implies $\mcal{V}(\pmb{\gamma},\pmb{\xi}) = -\tilde{\pmb{\lambda}}^{*\top} \pmb{\gamma} - \tilde{\pmb{\mu}}^{*\top} \pmb{\xi} + W$,
	where $W$ denotes additional terms not involving $\pmb{\gamma}$ or $\pmb{\xi}$ (cf.~\eqref{eq:dual_obj}). \blue{Lemma~\ref{lemma:fixLICQ} and Theorem~\ref{thm:LICQ} can be converted into a generalized test for strong duality for arbitrary C-OPF models beyond QC-OPF. \edit{It is sketched in the following remark and is demonstrated for CDF-OPF in~\ref{app:cdf-opf}.}
		\begin{remark}[Strong Duality Test for Arbitrary C-OPF]
            \label{rem:sd}
			For any C-OPF which can be written in the form~\eqref{prob:orig}, if the individual components of the functions $\mb{g}$, $\mb{h}$, $\tilde{\mb{g}}$, and $\tilde{\mb{h}}$ are smooth (i.e. they belong to class $C^\infty$) and the fixed constraints $\mb{g}$ and $\mb{h}$ satisfy Lemma~\ref{lemma:fixLICQ}, then strong duality holds almost everywhere on the open set $\Gamma \times \Xi$ representing the parameter space. Since this test only requires checking a rank condition for fixed constraints, it can be used to certify an C-OPF family at once.
	\end{remark}}
	LICQ (correspondingly, strong duality) ensures the existence and uniqueness of Lagrange multipliers satisfying the KKT system~\cite[Thm.~2]{GW:2013}. The closed form of $\nabla_{\pmb{\gamma}}\mcal{V}(\pmb{\gamma},\pmb{\xi})$ and $\nabla_{\pmb{\xi}}\mcal{V}(\pmb{\gamma},\pmb{\xi})$ follows immediately, as given below.
	
	\begin{prop}[Gradients of value function]
		\label{prop:grad}
		Given QC-OPF~\eqref{eq:qc-all}, 
		for almost every $(\pmb{\gamma},\pmb{\xi}) \in \Gamma \times \Xi$
		we have
		\begin{align}
			\nabla_{\pmb{\gamma}}\mcal{V}(\pmb{\gamma},\pmb{\xi}) = -\tilde{\pmb{\lambda}}^*,\quad \nabla_{\pmb{\xi}}\mcal{V}(\pmb{\gamma},\pmb{\xi}) = -\tilde{\pmb{\mu}}^*,
		\end{align}
		where $\tilde{\pmb{\lambda}}^*$ and $\tilde{\pmb{\mu}}^*$ are the unique optimal solution to problem~\eqref{eq:dual_problem}.
	\end{prop}
	\edit{An equivalent result holds for CDF-OPF as shown in~\ref{app:cdf-opf}, as well as any other C-OPF family which satisfy the criteria outlined in Remark~\ref{rem:sd}}. Since the optimal dual solution serves as the gradients of $\mcal{V}(\pmb{\gamma},\pmb{\xi})$, we can develop a customized neural network to learn these gradients and predict binding constraints for CS, as demonstrated in subsection~\ref{subsec:CS}.  
	Note that the dual solution-based DCOPF solvers (e.g.,~\cite{LC-YZ-BZ-2022}) can be considered as a special case of our framework.

	\subsection{Strict Complementarity}
	
	In this part, we justify using dual variables to remove non-binding constraints at the optimum by assuming the following condition for C-OPF.
	
	\begin{assumption}[Strict complementarity]
		\label{assmp:strict_comp}
		Given C-OPF~\eqref{prob:orig}, the strict complementarity condition holds for $\tilde{\mb{g}}$, i.e., for any optimal primal-dual solution we have
		\begin{align}
			\tilde{\pmb{\lambda}}^{*}_i>0, ~\mathrm{for~all}~i \in \mcal{A}_{\tilde{\mb{g}}}(\mb{x}^{*}).
		\end{align}
	\end{assumption}	
	The strict complementarity condition asserts that all binding inequality constraints have strictly positive multipliers (i.e., no degenerate inequalities). This condition \edit{has been assumed in previous works on OPF~\cite[Assumption 1]{MKS-VK-GBG-2022}}, and is observed empirically in our simulations. \edit{However, degeneracy has been observed to occur in certain C-OPF models used in energy markets~\cite{MM-markets}. In case the C-OPF is a conic convex problem, degeneracy can be removed by slightly perturbing the parameter and resolving. This stems from the fact that strict complementarity is a \emph{generic} property of conic convex problems~\cite[Theorem 5]{pataki2001generic}. We relegate a full analysis of CS with degenerate solutions to future works.}
	\begin{prop}[Constraint reduction equivalence]
		\label{lem:convex_exclusion}
		Suppose it is known \emph{a priori} that $\tilde{\pmb{\lambda}}^*_i = 0$ for $i\in\mcal{I}\subseteq [\tilde{L}]$. Then, replacing constraint $\tilde{\mb{g}}(\mb{x})\mleq \pmb{\gamma}$ in~\eqref{prob:orig} with $[\tilde{\mb{g}}(\mb{x})]_{[\tilde{L}]\setminus \mcal{I}}\mleq \pmb{\gamma}_{[\tilde{L}]\setminus \mcal{I}}$ results in an equivalent problem.
	\end{prop}
	The proof is straightforward. Any optimal solution \(\mb{x}^*\) of the original problem satisfies the KKT system for the reduced problem, and vice versa. Consequently, the optimal solutions to both convex problems are identical. It is worth noting that Proposition~\ref{lem:convex_exclusion} does not hold for nonconvex problems in general. Consider the problem
	$$
	\min\limits_{x\leq 0} \quad (4x^2+x+1)(x^2-1),$$ where LICQ holds for all points. The KKT system is 
	\begin{gather*}
		\nabla_x \left( (4x^2+x+1)(x^2-1) + \lambda x\right) = 0\\
		x \leq 0,\quad \lambda \geq 0,\quad \lambda x = 0.
	\end{gather*}
	We can verify that $x^*=-0.6267,\lambda^*=0$ 
	is the globally optimal solution. Excluding the constraint $x\leq 0$ given $\lambda^*=0$, the optimal solution to the unconstrained problem becomes $x^* =  0.6043$. Therefore, the reduced problem is not equivalent to the original one. 
	
	\subsection{Constraint Screening via MoGE}\label{subsec:CS}
	\begin{figure*}[tb!]
		\centering
		\begin{subfigure}[t]{0.49\textwidth}
			\centering
			\includegraphics[height=2cm]{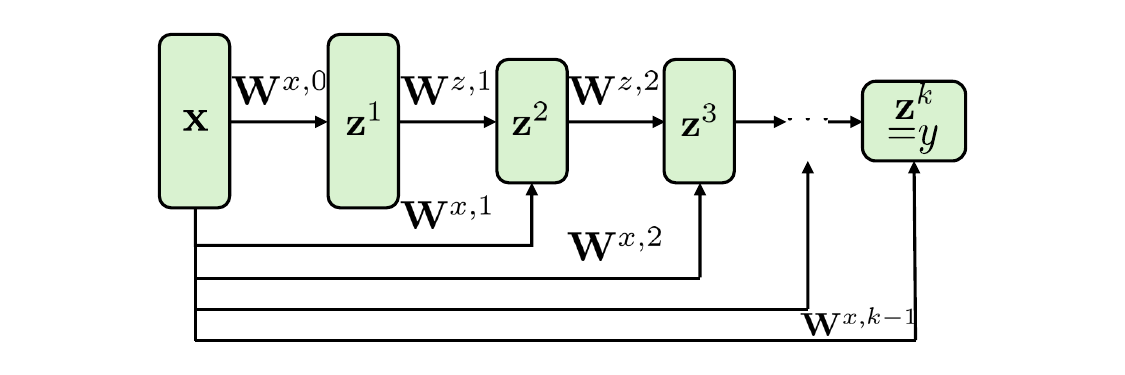}
			\caption{$\mb{y}=\icnn(\mb{x}|\pmb{\theta}_{icnn})$, with nonnegative weights $\{\mb{W}^{z,i}\}_{i=1}^{k-1}$ ensuring convexity of $\mb{y}$ in $\mb{x}$.}
		\end{subfigure}
		\hfill
		\begin{subfigure}[t]{0.49\textwidth}
			\centering
			\includegraphics[height=2cm]{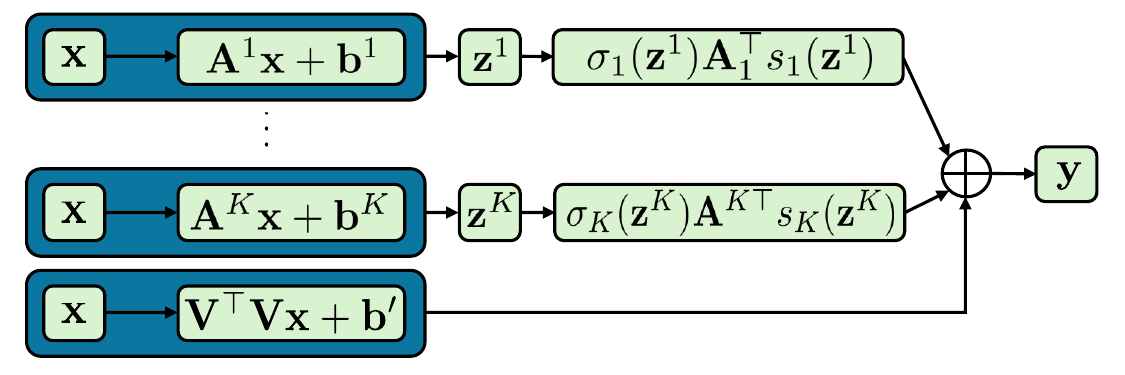}
			\caption{$\mb{y}=\mgn(\mb{x}|\pmb{\theta}_{mgn})$, where $\{\sigma_i(\,,\,)\}_{i=1}^K$ are convex, twice differentiable, nonnegative elementwise functions, and $\{s_i(\,.\,)\}_{i=1}^K$ are their respective derivatives.}
		\end{subfigure}
		\newline
		\hfill
		\begin{subfigure}[t]{0.49\textwidth}
			\centering
			\includegraphics[height=2cm]{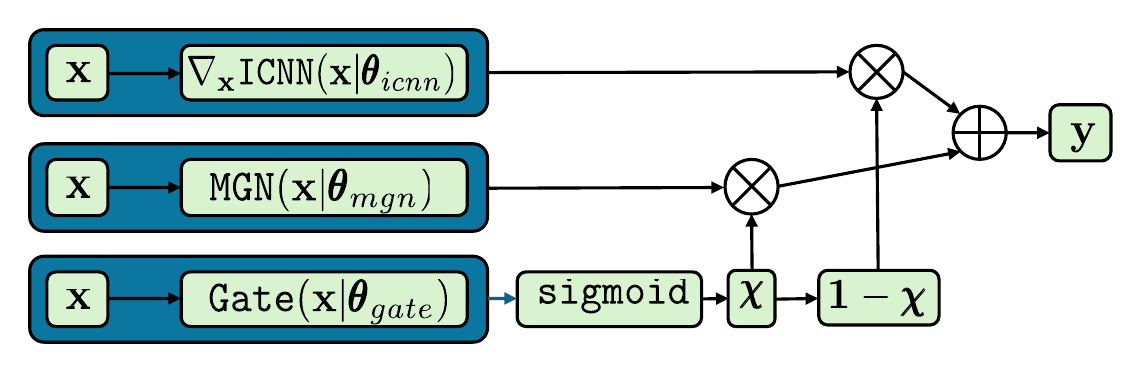}
			\caption{\moge\ architecture uses \icnn\ and \mgn\ alongside a gating network $\gate$, which is used to soft-select the better output from the two experts.}
		\end{subfigure}
		\caption{Schematic of the \moge\ architecture, alongside its primary constituent elements \icnn\ and \mgn.}
		\hfill
	\end{figure*}
	
	{
		In this section, we present the acceleration method for C-OPF by removing non-binding constraints from $\tilde{\mb{g}}(\mb{x})\mleq \pmb{\gamma}$. Our strategy involves training a model 
		which learns the gradients of $ \mcal{V}(\pmb{\gamma},\pmb{\xi})$, i.e. the mapping $(\pmb{\gamma},\pmb{\xi})\mapsto (-\tilde{\pmb{\lambda}}^*,-\tilde{\pmb{\mu}}^*)$ in supervised learning fashion. Since this mapping is the gradient of the value function (Proposition~\ref{prop:grad}), it motivates the ``gradient'' terminology used throught this paper. Any new problem instance is sped up by removing non-binding constraints identified from the predicted value of $\tilde{\pmb{\lambda}}^*$ using Assumption~\ref{assmp:strict_comp}.
		
		We aim to use machine learning models suitable for approximating gradients of convex functions, and consider two possible approaches. The first approach, used by input convex neural networks (ICNNs)~\cite{BA-LX-JZK:2017}, approximates the scalar-valued mapping $(\pmb{\gamma},\pmb{\xi}) \mapsto \mcal{V}(\pmb{\gamma},\pmb{\xi})$. Gradients are then generated by calculating $\nabla \mcal{V}$ using automatic differentiation (AD) during both training and inference.
		The second approach, utilized by monotone gradient networks (MGNs)~\cite{SC-SP-JMFM:2023}, directly learns the mapping $(\pmb{\gamma},\pmb{\xi})\mapsto (-\tilde{\pmb{\lambda}}^*,-\tilde{\pmb{\mu}}^*)$ without using AD. 
		
		A $K$-layer ICNN, denoted as $y = \mb{z}^K = \icnn(\mb{x}|\pmb{\theta}_{icnn})$, uses the architecture
		\begin{align*}
			&\mb{z}^{k+1} = \sigma_k \left(\mb{W}^{z,k} \mb{z}^k + \mb{W}^{x,k} \mb{x} + \mb{b}^k \right),\, k = 0,\dots,K-1,
		\end{align*}
		where $\mb{z}^i$ denotes the layer activations with $\mb{z}^0\equiv\mb{0}$ and $\mb{W}^{z,0}\equiv\mb{0}$. $\sigma_i$ are non-linear activation functions and $\pmb{\theta}_{icnn}=\{\mb{W}^{x,0:K-1},\mb{W}^{z,1:K-1},\mb{b}^{0:K-1}\}$ are the learnable parameters.
		If all weights $\mb{W}^{z,1:K-1}$ are nonnegative and all functions $\sigma_i$ are convex and non-decreasing, then \icnn\ is convex in its input-to-output mapping~\cite[Prop. 1]{BA-LX-JZK:2017}. Consequently, $\nabla_{\mb{x}}\icnn(\mb{x}|\pmb{\theta}_{icnn})$ can be used to learn the dual variables.
		On the other hand, \mgn\ is a shallow architecture with $K$ units given as
		\begin{align*}
			&\mb{y} = \mgn(\mb{x}|\pmb{\theta}_{mgn}) = \mb{b}' + \mb{V}^\top\mb{V}\mb{x} + \sum_{k=1}^K \sigma_k(\mb{z}^k)\mb{A}^{k\top} s_k(\mb{z}^k),\\
			& \mb{z}^k = \mb{A}^k\mb{x}+\mb{b}^k \;\forall k\in[K],\;\; \pmb{\theta}_{mgn} = \left\{ \mb{A}^{1:K},\mb{b}^{1:K},\mb{V},\mb{b}' \right\}.
		\end{align*}
		Given convex, nonnegative, and twice differentiable elementwise activations $\sigma_k$ with $s_k$ being $\sigma_k$'s elementwise derivative, it can be shown that the Jacobian of \mgn\ is positive semi-definite~\cite[Prop. 2]{SC-SP-JMFM:2023}. Therefore, \mgn\ is suitable for directly learning the gradients of convex functions.
		
		In Section~\ref{sec:4}, ablative testing will reveal limitations in the ability of either of these two models to effectively learn dual variables. To address this, we adopt a mixture-of-experts (MoE) architecture named \moge\ (Mixture of Gradient Experts). This approach integrates the strengths of both models using an auxiliary gating network, aiming to enhance robustness and performance in learning dual variables. MoE architectures, pioneered over three decades ago~\cite{RAJ-etal:1991}, are increasingly favored for their capability to leverage diverse model strengths effectively. Recent applications in large language models like Mixtral~\cite{MIXTRAL} underscore their suitability for complex tasks requiring nuanced learning and adaptation.

		Letting $\pmb{\chi}$ denote the elementwise sigmoid function, the overall output of \moge\ is
		\begin{align*}
			\mb{y} = \pmb{\chi}\left(\gate (\mb{x}|\pmb{\theta}_{gate})\right)\icnn(\mb{x}|\pmb{\theta}_{icnn}) 
			+ \left( \mb{1}- \pmb{\chi}\left(\gate (\mb{x}|\pmb{\theta}_{gate})\right)\right) \mgn(\mb{x}|\pmb{\theta}_{mgn}).
		\end{align*}
		The training process involves pre-training $\icnn$ and $\mgn$ for a certain number of steps, followed by training $\gate$ on the outputs of the trained experts for the same number of epochs. We consider the mean-square error loss for the $\tth{i}$ data sample
		\begin{align}
			l^{(i)} \ldef 
			\left\Vert [\nn(\pmb{\gamma}^{(i)},\pmb{\xi}^{(i)})]_{(
				\mcal{I}_{\tilde{\pmb{\lambda}}},\mcal{I}_{\tilde{\pmb{\mu}}})}- \left(-\tilde{\pmb{\lambda}}^{(i)*}, -\tilde{\pmb{\mu}}^{(i)*}\right) 
			\right\Vert_2^2,
		\end{align}
		where $(
		\mcal{I}_{\tilde{\pmb{\lambda}}},\mcal{I}_{\tilde{\pmb{\mu}}})$ are indices corresponding to $\tilde{\pmb{\lambda}}^*$ and $\tilde{\pmb{\mu}}^*$.
		During the training of $\gate$, a weighing factor of $100N$ is used for loss indices that are ground truth binding, thereby allowing better learning of sparse binding constraints. The $\gate$ network acts as a polling mechanism which soft-selects the response of the more accurate expert as the final response.} \blue{We conclude this section by highlighting the necessity of using the mixture of two different models.
		
		\begin{remark}[Necessity of Using a Mixture-of-Experts Model]
            \label{rem:necessity-moge}
			As will be seen in Algorithm~\ref{alg:accelerate} and Proposition~\ref{prop:two-iters}, in order to achieve a C-OPF solution which is identical to the full problem, we must solve screened instances of the C-OPF at most $p+1$ times, where $p$ denotes the number of false negatives, i.e. ground truth binding constraints which are classified as non-binding by the given CS method. In order to reduce solve time, p must be 0 on average, which necessitates that false negative rates must be brought close to zero. \icnn, while being competetive in terms of correctly classifying constraints, still suffers from unacceptably high rates of false negatives as seen in Section~\ref{sec:4}. On the other hand \mgn\ tends to generally classify a large number of constraints as binding by the virtue of it being less performant due to its shallow architecture. This behavior of \mgn\ can be used to counteract \icnn's false negatives by combining the two using \gate, which is a strategy that has been previously used to extract enhanced performance out of a combination of biased experts~\cite{AA-YA-2020}.
	\end{remark}}

	\subsection{Dataset Generation and Training}
	
	\begin{algorithm}[t!]
		\caption{Dataset Generation and Training $\nn$}
		\label{alg:dataset}
		\textbf{Input:} Upper bound $(\bar{\pmb{\gamma}},\bar{\pmb{\xi}})$ and lower bound $(\underline{\pmb{\gamma}},\underline{\pmb{\xi}})$, number of initial samples $K_1$, number of convex hull samples $K_2$\\
		\textbf{Output:}  Trained $\nn$
		\begin{algorithmic}[1]
			\State $\mcal{D} = \{\},\mcal{D}_1 = \{\},\mcal{D}_2=\{\}$ 
			\For{$i = 1$ to $K_1$} \Comment{\texttt{Generate 1st dataset}}
			\State Sample $(\pmb{\gamma},\pmb{\xi}) \sim \mathrm{Unif} ((\underline{\pmb{\gamma}},\underline{\pmb{\xi}}),(\bar{\pmb{\gamma}},\bar{\pmb{\xi}}))$
			\State Solve~\eqref{prob:orig} with parameters $(\pmb{\gamma},\pmb{\xi})$
			\If {\eqref{prob:orig} is infesible for $(\pmb{\gamma},\pmb{\xi})$}, continue
			\Else
			\State Record inputs \& outputs $\mb{d} \ldef (\pmb{\gamma},\pmb{\xi},\tilde{\pmb{\lambda}}^*,\tilde{\pmb{\mu}}^*)$
			\State $\mcal{D}_1 = \mcal{D}_1 \cup \{\mb{d}\}$
			\EndIf
			\EndFor
			\For {$i=1$ to $K_2$} \Comment{\texttt{Generate 2nd dataset}}
			\State Randomly generate a point $(\pmb{\gamma},\pmb{\xi}) \in \mathrm{conv}\{(\pmb{\gamma}^{(1)},\pmb{\xi}^{(1)}),\dots, (\pmb{\gamma}^{(|\mcal{D}_1|)},\pmb{\xi}^{(|\mcal{D}_1|)})\}$
			\State Solve~\eqref{prob:orig} with the parameters $(\pmb{\gamma},\pmb{\xi})$
			\State Record inputs \& outputs $\mb{d} \ldef (\pmb{\gamma},\pmb{\xi},\tilde{\pmb{\lambda}}^*,\tilde{\pmb{\mu}}^*)$
			\State $\mcal{D}_2 = \mcal{D}_2 \cup \{\mb{d}\}$
			\EndFor
			\State $\mcal{D} = \mcal{D}_1 \cup \mcal{D}_2$ 
			\State Train $\nn$ using minibatch training on losses $\{l^{(i)}\}_{i\in[|\mcal{D}|]}$
			\State \textbf{return} $\nn$
		\end{algorithmic}
	\end{algorithm}
	
	\begin{algorithm}[t!]
		\caption{Constraint Screening based Accelerated C-OPF}
		\label{alg:accelerate}
		\textbf{Input:} Trained $\nn$, feasible parameters $(\pmb{\gamma},\pmb{\xi})$\\
		\textbf{Output:}  Solution $\mb{x}^*$ to C-OPF~\eqref{prob:orig}
		\begin{algorithmic}[1]
			\State $\tilde{\pmb{\lambda}}^* = [\nn(\pmb{\gamma},\pmb{\xi})]_{\mcal{I}_{\tilde{\pmb{\lambda}}}}$
			\Comment{\texttt{Predict optimal dual vars.}}
			\State $\mcal{A}_\text{bind} = \{ i: \tilde{\pmb{\lambda}}^*_i \neq 0 \}$ \Comment{\texttt{Set of binding constr.}} \label{line:proclaim}
			\State $\mcal{A}_\text{viol} = \varnothing$ \Comment{\texttt{Set of violated constr.}}
			\Do
			\State Replace $\tilde{\mb{g}}(\mb{x})\mleq \pmb{\gamma}$ in~\eqref{prob:orig} with $[\tilde{\mb{g}}(\mb{x})]_{\mcal{A}_\text{bind}} \mleq [\pmb{\gamma}]_{\mcal{A}_\text{bind}}$
			\State Solve the reduced version of~\eqref{prob:orig}		
			\State Save the optimal solution $\mb{x}^*$ 
			\State Set $\mcal{A}_\text{viol} = \{i: \tilde{\mb{g}}_i(\mb{x}^*) > \pmb{\gamma}_i \}$ 
			\State Update $\mcal{A}_\text{bind} = \mcal{A}_\text{bind} \cup \mcal{A}_\text{viol}$ 
			\doWhile{$\mcal{A}_\text{viol} \neq \varnothing$}
			\State \textbf{return} $\mb{x}^*$
		\end{algorithmic}
	\end{algorithm}
	
	The dataset generation and $\nn$ training process are highlighted in Algorithm~\ref{alg:dataset}. As seen in Table~\eqref{tab:convert_to_param}, the parameters $(\pmb{\gamma},\pmb{\xi})$ contain values such as real and reactive load demands and thermal limits. To enhance the robustness of CS to newly encountered values of such parameters, we first create dataset $\mcal{D}_1$ by uniformly sampling $\pmb{\gamma}$ and $\pmb{\xi}$ over given intervals. Each dimension of the parameter space is sampled independently. Then, problem~\eqref{prob:orig} is solved with each of the sampled parameters. The resulting parameters $(\pmb{\gamma},\pmb{\xi})$ and optimal dual variables $\tilde{\pmb{\lambda}}^{*}, \tilde{\pmb{\mu}}^{*}$ are all added to $\mcal{D}_1$. If any parameter causes an infeasible problem, it is ignored. Three points are noteworthy. First, the sampling process can be replaced with the use of historic datasets from solving prior C-OPF problems. Second, employing the aforementioned sampling approach for generating the entire dataset could lead to numerous sampled parameters being infeasible, causing a significant waste of time. Therefore, we sample a distribution derived from the existing data to bootstrap the generation of additional feasible parameters (i.e. the set $\mcal{D}_2$ in Algorithm~\ref{alg:dataset}). \blue{Third, the interior of the convex hull of $\mcal{D}_1$ forms the open set $\Gamma \times \Xi$ discussed in theoretical analyses.}

	
	\subsection{Accelerated C-OPF with CS} Following the training of $\nn$, we use it to accelerate C-OPF as detailed in Algorithm~\ref{alg:accelerate}. Given a new instance of $(\pmb{\gamma},\pmb{\xi})$, $\nn$ predicts the value of  $\tilde{\pmb{\lambda}}^*$. 
	Based on Assumption~\ref{assmp:strict_comp}, 
	non-zero elements of $\tilde{\pmb{\lambda}}^*$ indicate the corresponding constraints must be binding. False positive mispredictions (classifying a non-binding constraint as binding) merely increase the reduced problem size without affecting the optimal solution. However, false negative mispredictions (classifying a binding constraint as non-binding) can alter the optimal solution. In such cases, Algorithm~\ref{alg:accelerate} adds the violated constraints to the list of binding constraints and re-solves C-OPF. The following result {characterizes the number of re-solves as a function of model error}.
	
	\begin{prop}
		\label{prop:two-iters}
		The do-while loop in Algorithm~\ref{alg:accelerate} terminates after at most $p+1$ iterations, where $p$ is the number of false negative constraints (constraints which bind but are classified by \nn\ as non-binding in line~\ref{line:proclaim} of Algorithm~\ref{alg:accelerate}).
	\end{prop}
	
	{Proposition~\ref{prop:two-iters} outlines the worst-case scenario and is proved in~\ref{app:maxloops}. However, in practice, as shown in Section~\ref{sec:4}, our experiments consistently detect any violated constraints, if present, during the very first iteration of the do-while loop. Furthermore, re-solve is not even necessary for the majority of instances.} These facts together ensure that the resulting C-OPF solve time can be significantly reduced. We also note that Algorithm~\ref{alg:accelerate} is compatible with CS methods beyond $\nn$.	
	\begin{remark}[Generality of Algorithm~\ref{alg:accelerate}]
		\label{rem:generality}
		In Algorithm~\ref{alg:accelerate}, the sole role of $\nn$ is to generate the optimal dual variable $\tilde{\pmb{\lambda}}^*$ for identifying binding constraints. Therefore, Algorithm~\ref{alg:accelerate} can be used with any alternative method capable of identifying binding constraints as a function of $(\pmb{\gamma},\pmb{\xi})$. Furthermore, Prop.~\ref{prop:two-iters} continues to hold for CS methods other than $\nn$, since it is only based on the properties of the C-OPF (\emph{viz.} convexity, Prop.~\ref{prop:grad} and Assumption~\ref{assmp:strict_comp}).
	\end{remark}

    \edit{
    We conclude this section with a result on the generalization capability of \moge\ which extends~\cite[Theorem 6.4]{LC-YZ-BZ-2022} to the \moge\ architecture.

    \begin{thm}[Generalization Error of \moge]
    \label{thm:moge-gen}
        Let $\mcal{A} \ldef \{\allowbreak(\pmb{\gamma}^{(1)},\pmb{\xi}^{(1)}),\allowbreak\cdots,\allowbreak(\pmb{\gamma}^{(K)},\pmb{\xi}^{(K)})\}\allowbreak \subseteq\allowbreak \Gamma \times \Xi$ and $\nabla_{\pmb{\gamma},\pmb{\xi}} \mcal{V}(\pmb{\gamma},\pmb{\xi}) = (-\tilde{\pmb{\lambda}}^*,-\tilde{\pmb{\mu}}^*)$ for all $ (\pmb{\gamma},\pmb{\xi})\in\mcal{A}$. Further, let $\mcal{G}(\pmb{\lambda},\pmb{\xi})$ and $\mcal{H}(\pmb{\lambda},\pmb{\xi})$, both defined on $\Gamma \times \Xi$, be parameterized monotone functions which are convex gradients\footnote{It is possible for a function $F:\real{n}\mapsto \real{n}$ to be monotone, yet for there to not exist a convex function $\phi$ such that $F\equiv \nabla \phi$ over $\text{dom}(\phi)$. \icnn\ and \mgn\ avoid this problem: they are gradients of a convex field by construction.}. Lastly, let $\mcal{T}(\pmb{\lambda},\pmb{\xi}) \in [0,1]$ over $\Gamma \times \Xi$ and let $\mcal{Z} \ldef \mcal{TG}+(1-\mcal{T})\mcal{H}$ (arguments omitted). If $\mcal{G}$ and $\mcal{H}$ agree with $\nabla \mcal{V}$ on all points in $\mcal{A}$, then for all points $(\pmb{\gamma},\pmb{\xi})\in\text{conv}(\mcal{A})$,
        \begin{align*}
            \nabla_{\pmb{\gamma},\pmb{\xi}} \mcal{V}(\pmb{\gamma},\pmb{\xi}) = \mcal{G}(\pmb{\gamma},\pmb{\xi}) = \mcal{H}(\pmb{\gamma},\pmb{\xi}) = \mcal{Z}(\pmb{\gamma},\pmb{\xi}) = (-\tilde{\pmb{\lambda}}^*,-\tilde{\pmb{\mu}}^*).
        \end{align*}
    \end{thm}
    The proof follows from standard properties of convex functions and is provided in~\ref{app:moge-gen-proof}. 
    \begin{remark}[\moge\ Performance Improves with Dense Sampling]
    \label{cor:moge-perf}
        Theorem~\ref{thm:moge-gen} colloquially states that if \moge\ is trained to zero error over the training dataset $\mcal{D}$, then it achieves zero test error over the convex hull of any subset of $\mcal{D}$ such that $\nabla \mcal{V}$ is constant over said subset. This observation motivates the idea that new samples from regions of $\Gamma \times \Xi$ with good sample coverage in the training dataset $\mcal{D}$ enjoy more accurate outputs. This is because for cases where $\mcal{A}$ belongs to a region of $\Gamma \times \Xi$ over which $\mcal{V}$ is smooth and the points in $\mcal{A}$ are clustered close together, the outputs of $\nabla \mcal{V}$ are approximately equal for all points in $\mcal{A}$.
     \end{remark}
    }
	
	\section{Simulations}
	\label{sec:4}
	\newcommand{\minp}[1]{
		\begin{minipage}[b]{0.2\textwidth}
			\vfill
			{#1}
			\vfill
		\end{minipage}
	}
	\blue{
		\begin{table}[tb!]
			\centering
			\caption{Number of elements in different cases from the PGLIB and MATPOWER libraries, and solve times for the full problem.}
			\resizebox{\linewidth}{!}{
				\begin{tabular}{|l|c|c|c|c|c|}
					\hline
					\textbf{Cases} & \textbf{Buses} & \textbf{Lines} & \textbf{Generators} & \textbf{Transformers} & \makecell{\textbf{Solve Time}\\\textbf{(Full Problem)}}\\
                    \hline
                    \multicolumn{6}{|c|}{PGLIB - Meshed}\\
					\hline \texttt{case118} & 118 & 186 & 54 & 11 & 0.13s\\
					\texttt{case1354} & 1354 & 1991 & 260 & 240 & 4.81s\\
					\texttt{case2312} & 2312 & 3013 & 226 & 857 & 8.21s\\
					\texttt{case4601} & 4601 & 7199 & 408 & 2054 & 21.56s\\
					\texttt{case10000} & 10000 & 13139 & 2016 & 2374 & 49.52s\\
					\hline
                    \multicolumn{6}{|c|}{MATPOWER - Radial} \\
                    \hline
                    \texttt{case136} & 136 & 135 & 1 & 0 & 2.74s\\
                    \texttt{case1197} & 1197 & 1196 & 1 & 0 & 6.22s\\
                    \hline
				\end{tabular}
			} 
			\label{tab:testcases}
		\end{table}
	}	
	\begin{figure*}[tb!]
		\centering
		\includegraphics[width=0.93\linewidth]{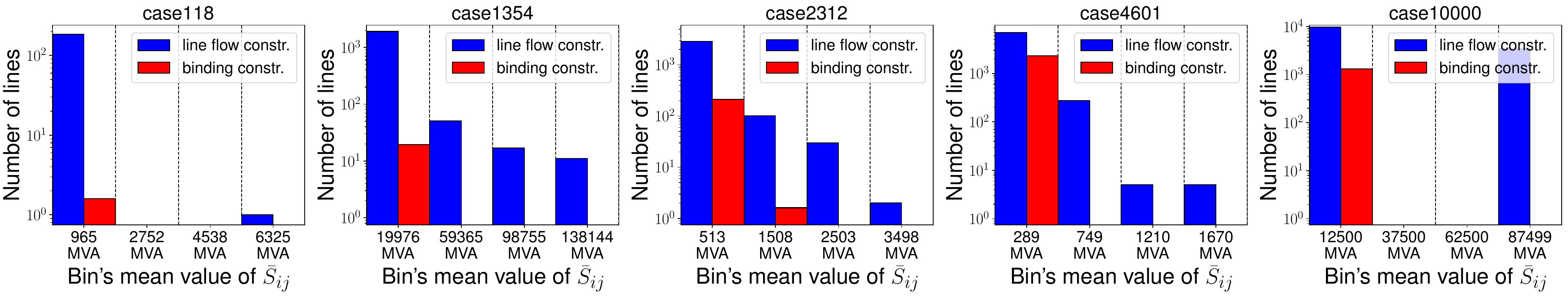}
		\caption{Histograms of thermal limits, along with the average number (across the test set and forward/reverse flows) of corresponding line flow constraints which bind for different cases in PGLIB. \blue{These constraints, which constitute $\tilde{\mb{g}}(\mb{x})\mleq\pmb{\gamma}$, are candidates for removal via presolve.}}
		\label{fig:binding}
	\end{figure*}
	{ 
		\begin{table*}[h!]
			\centering
			\caption{Performance evaluation of constraint classification approaches based on average solve time, training time, and fraction of re-solves for 100 test samples for QC-OPF. A negative reduction in solve time indicates an increase in solve time over the original problem.}
			\resizebox{\linewidth}{!}{
\renewcommand{\b}[1]{\textbf{#1}}
\begin{tabular}{|l|c|c|c|c|c|c|c|}
	\hline
	\textbf{\makecell{Method of\\Constr. Removal}} & \makecell{\textbf{Training}\\\textbf{Time}} & \multicolumn{4}{c|}{\textbf{\makecell{Confusion Matrix\\ (Full Test Set)}	}} & \textbf{\makecell{Percent Reduction\\in Solve Time\\(100 samples)}} & \textbf{\makecell{Fraction\\Resolved\\(100 samples)}}\\
	\hline
	& &\makecell{\textbf{True}\\\textbf{Positive}} & \makecell{\textbf{True}\\\textbf{Negative}} & \makecell{\textbf{False}\\\textbf{Positive}} & \makecell{\textbf{False}\\\textbf{Negative}} & & \\
	\hline
	\multicolumn{8}{|c|}{\texttt{case118}}\\
	\hline
	Original problem & - & 0.86\% & 99.14\% & - & - & 0.00\% & -\\
	\moge\  & 42.9275s & \b{0.86\%} & 94.35\% & 4.79\% & \b{0.00\%}  & 22.01\% & \b{0.00}\\
	\icnn\  & 17.7126s & 0.83\% & 95.68\% & 3.47\% & 0.02\%  & \b{22.16\%} & \b{0.00}\\
	\mgn\  & 2.1961s & 0.83\% & 4.88\% & 94.26\% & 0.02\% & -32.23\% & \b{0.00}\\
	Deep classifier & 3.7815s & 0.51\% & \b{99.12\%} & \b{0.02\%} & 0.34\% & -27.61\% & 0.69\\
	Ridge regression & 3.3604s & 0.52\% & 99.10\% & 0.04\% & 0.33 & -32.46\% & 0.63\\
	XGBoost & \b{1.5327s} & 0.51\% & 99.12\% & 0.03\% & 0.34\%  & -38.39\% & 0.70\\
	\hline
	\multicolumn{8}{|c|}{\texttt{case1354}}\\
	\hline
	Original problem & - & 0.97\% & 99.03\% & - & - & 0.00\% & -\\
	\moge\  & 28.1495s & \b{0.97\%} & 97.24\% & 1.79\% & \b{4e-5\%} & \b{33.89\%} & \b{0.00}\\
	\icnn\  & 12.8741s & 0.97\% & 97.64\% & 1.39\% & 6e-4\% & 28.46\% & 0.09\\
	\mgn\  & 11.4738s & 0.95\% & 3.24\% & 95.79\% & 0.02\% & -2.62\% & 0.10\\
	Deep classifier & 7.5701s & 0.92\% & \b{98.99\%} & \b{0.04\%} & 0.05\% & -3.57\% & 0.60\\
	Ridge regression & \b{6.4307s} & 0.92\% & 98.99\% & 0.04\% & 0.05\% & -3.40\% & 0.60\\
	XGBoost & 33.3336s & 0.90\% & 99.01\% & 0.02\% & 0.07\% & -26.15\% & 0.97\\
	\hline
	\multicolumn{8}{|c|}{\texttt{case2312}}\\
	\hline
	Original problem & - & 7.14\% & 92.86\% & - & - & 0.00\% & -\\
	\moge\  & 49.6899s & \b{7.13\%} & 87.17\% & 5.69\% & \b{3e-4\%} & \b{25.38\%} & \b{0.00}\\
	\icnn\  & 16.6255s & 7.07\% & 89.91\% & 2.96\% & 0.07\% & 22.28\% & 0.06\\
	\mgn\  & 25.6937s & 7.12\% & 0.45\% & 92.42\% & 0.02\% & 14.99\% & \b{0.00}\\
	Deep classifier & \b{11.0422s} & 6.92\% & 92.62\% & 0.25\% & 0.21\% & -44.77\% & 1.00\\
	Ridge regression& 12.9950s & 7.04\% & 91.57\% & 1.30\% & 0.09\% & -39.94\% & 0.92\\
	XGBoost & 90.4948s & 6.85\% & \b{92.67\%} & \b{0.19\%} & 0.29\% & -45.04\% &1.00\\
	\hline
	\multicolumn{8}{|c|}{\texttt{case4601}}\\
	\hline
	Original problem & - & 31.93\% & 68.07\% & - & - & 0.00\% & -\\
	\moge\  & 171.3697s & \b{31.93\%} & 58.49\% & 9.57\% & \b{0.00\%} & \b{ 20.15\%} & \b{0.00}\\
	\icnn\  & 26.6654s & 31.76\% & 59.28\% & 8.79\% & 0.17\% & 19.72\% & \b{0.00}\\
	\mgn\  & 101.7099s & 31.87\% & 3.22\% & 64.85\% & 0.06\% & 18.02\% & \b{0.00}\\
	Deep classifier & \b{21.4779s} & 31.91\% & 59.24\% & 8.83\% & 0.02\% & 18.99\% &0.03\\
	Ridge regression & 40.9736s & 31.91\% & 59.24\% & 8.83\% & 0.02\% & 19.22\% &0.03\\
	XGBoost & 569.9688s & 31.76\% & \b{59.32\%} & \b{8.75\%} & 0.17\% & 18.99\% & 0.03\\
	\hline
	\multicolumn{8}{|c|}{\texttt{case10000}}\\
	\hline
	Original problem & - & 9.94\% & 90.06\% & - & - & 0.00\% & -\\
	\moge\  & 580.2680s & \b{9.94\%} & 86.99\% & 3.07\% & \b{0.00\%} & \b{24.53\%} & \b{0.00}\\
	\icnn\  & \b{48.1584s} & 8.87\% & 87.61\% & 2.44\% & 1.07\% & 20.21\%  & 0.02\\
	\mgn\  & 363.8684s & 9.94\% & 26.54\% & 63.52\% & 0.00\% & 8.76\% & \b{0.00}\\
	Deep classifier & 72.0453s & 9.93\% & \b{87.31\%} & \b{2.75\%} & 0.01\% & 8.78\% & 0.17\\
	Ridge regression & 197.7062s & 9.93\% & 87.31\% & 2.75\% & 0.01\% & 8.76\% &0.17\\
	XGBoost & -OOT- &- & - & - & - & - & -\\\hline
\end{tabular}
			}
			\label{tab:presolve_performance}
		\end{table*}

        \begin{table*}[h!]
			\centering
			\caption{Performance evaluation of constraint classification approaches based on average solve time, training time, and fraction of re-solves for 100 test samples for CDF-OPF.}
			\resizebox{\linewidth}{!}{
\renewcommand{\b}[1]{\textbf{#1}}
\begin{tabular}{|l|c|c|c|c|c|c|c|}
	\hline
	\textbf{\makecell{Method of\\Constr. Removal}} & \makecell{\textbf{Training}\\\textbf{Time}} & \multicolumn{4}{c|}{\textbf{\makecell{Confusion Matrix\\ (Full Test Set)}	}} & \textbf{\makecell{Percent Reduction\\in Solve Time\\(100 samples)}} & \textbf{\makecell{Fraction\\Resolved\\(100 samples)}}\\
	\hline
	& &\makecell{\textbf{True}\\\textbf{Positive}} & \makecell{\textbf{True}\\\textbf{Negative}} & \makecell{\textbf{False}\\\textbf{Positive}} & \makecell{\textbf{False}\\\textbf{Negative}} & & \\
	\hline
	\multicolumn{8}{|c|}{\texttt{case136}}\\
	\hline
	Original problem & - & 86.53\% & 13.47\% & - & - & 0.00\% & -\\
	\moge\  & 17.3123s & \b{86.53\%} & 9.85\% & 3.62\% & \b{0.00\%}  & \b{9.87\%} & \b{0.00}\\
	\icnn\  & 9.2234s & 86.45\% & 10.72\% & 2.75\% & 0.08\%  & 9.12\% & 0.04\\
	\mgn\   & 3.1643s  & 86.50\% & 6.07\% & 7.40\% & 0.03\% & 7.22\% & 0.01\\
	Deep classifier & 4.2356s & 85.91\% & 13.36\% & 0.11\% & 0.62\% & -14.51\% & 0.86\\
	Ridge regression & 3.9744s & 85.88\% & 13.34\% & 0.13\% & 0.65\% & -13.95\% & 0.84\\
	XGBoost & \b{2.0030s} & 85.96\% & \b{13.37\%} & \b{0.10\%} & 0.57\%  & -17.80\% & 0.92\\
	\hline
	\multicolumn{8}{|c|}{\texttt{case1197}}\\
        \hline
	Original problem & - & 64.38\% & 35.62\% & - & - & 0.00\% & -\\
	\moge\  & 31.6432s & \b{64.38\%} & 33.02\% & 2.60\% & \b{0.00\%} & \b{26.41\%} & \b{0.00}\\
	\icnn\  & 13.9122s & 64.32\% & 33.52\% & 2.10\% & 0.06\% & 23.84\% & 0.05\\
	\mgn\   & 17.8436s & 64.34\% & 17.12\% & 18.50\% & 0.04\% & 12.36\% & 0.01\\
	Deep classifier & \b{8.5220s} & 64.08\% & 35.32\% & 0.30\% & 0.30\% & -4.22\% & 0.40\\
	Ridge regression & 9.9355s & 64.06\% & 35.31\% & 0.31\% & 0.32\% & -5.91\% & 0.37\\
	XGBoost & 41.7673s & 64.11\% & \b{35.39\%} & \b{0.23\%} & 0.27\% & -2.13\% & 0.72\\
	\hline
\end{tabular}

			}
			\label{tab:presolve_performance-cdf}
		\end{table*}
		\edit{We test the dataset generation Algorithm~\ref{alg:dataset} and the accelerated C-OPF in Algorithm~\ref{alg:accelerate}, focusing on the QC-OPF~\eqref{eq:qc-all} and CDF-OPF models~\eqref{eq:cdf-opf}. We utilize the PGLIB and MATPOWER libraries for our test cases, with the former containing meshed transmission networks and the latter containing radial distribution networks. We select seven specific test cases which are detailed in Table~\ref{tab:testcases}.} These test cases prescribe the grid topology, line and transformer parameters, nominal load demands, thermal limits, and PAD limits. 

        \edit{
        \begin{remark}[Dataset Generation]
        \label{rem:dset-gen}
        Variability in parameters $(\pmb{\gamma}, \pmb{\xi})$ arises from perturbations around the prescribed nominal values of $P^d_i$ and $Q^d_i$ at each bus. Specifically, if the nominal values of the aforementioned quantities are $\tilde{P}^d_i$ and $\tilde{Q}^d_i$, then they are sampled i.i.d elemtwise in the boxes $[0.25\tilde{P}^d_i,1.75\tilde{P}^d_i]$ and $[0.25\tilde{Q}^d_i,1.75\tilde{Q}^d_i]$ respectively. On the other hand, the nominal values of thermal limits $\bar{S}_{ij}$ are slightly perturbed per-instance in the range $[(1-\epsilon)\bar{S}_{ij},(1+\epsilon)\bar{S}_{ij}]$ to ensure that the set $\Gamma \times \Xi$ has a nonzero measure (note that $\bar{S}_{ij}$ is replaced with $\bar{S}_i$ for the CDF-OPF model). We want to highlight that $\Gamma \times \Xi$ having a nonzero measure (i.e. no dimensions of $(\pmb{\gamma},\pmb{\xi})$ being constant across the dataset) is important for ensuring strong duality guarantees in Theorem~\ref{thm:LICQ} hold.\par
        The aforementioned sampling principle is used to generate 1000 values of $(\pmb{\gamma},\pmb{\xi})$ to generate the dataset $\mcal{D}_1$, with infeasible instances discarded (see Algorithm~\ref{alg:dataset}). Convex hull sampling is then used to create the augmented dataset $\mcal{D}_2$ such that $|\mcal{D}|=|\mcal{D}_1\cup\mcal{D}_2|=5000$. The first 4000 points in $\mcal{D}$ are used for training, with the remaining points reserved for the test set. 
        \end{remark}}
		
		We use the Python implementation of the IPOPT solver~\cite{IPOPT}, \texttt{cyipopt}, to solve optimization problems. This solver is augmented with customized code to ensure efficient solutions out of the box. \texttt{cyipopt} provides optimal primal-dual solutions, with the dual solutions used for subsequent simulation aspects.
		
		\subsection{Speeding up QC-OPF in \texttt{cyipopt}}
		\label{sec:cyipopt}
		As a Python implementation of IPOPT, \texttt{cyipopt} lacks the speed advantages of Julia's JIT-compiled IPOPT implementations. Nevertheless, we can achieve comparable speed improvements in Python by pre-computing coefficients as outlined herein. All implementations of IPOPT solve the problem
		\begin{align*}
			\min\limits_{\mb{x}\in\real{n}} \quad& F(\mb{x})\\
			\text{s.t.} \quad& \underline{\mb{C}} \preceq \mb{C}(\mb{x}) \preceq \bar{\mb{C}},\quad\underline{\mb{x}} \preceq \mb{x} \preceq \bar{\mb{x}}
		\end{align*}
		wherein $\mb{C}:\real{n}\mapsto\real{p}$ are the lumped constraint functions, with equality constraints modeled by setting $\underline{\mb{C}}_i=\bar{\mb{C}}_i$. Given $\mb{x}\in\real{n}$, $\pmb{\alpha}\in\real{p}$, and $\beta\in\real{}$, the user is required to provide routines that numerically compute the quantities $F(\mb{x})$, $\mb{C}(\mb{x})$, $\nabla F(\mb{x})$, $\nabla \mb{C}(\mb{x})$, and $\mb{H}(\mb{x},\pmb{\alpha},\beta):=\beta\nabla^2 F(\mb{x}) + \sum_{j=1}^p \alpha_j\nabla^2\mb{C}_j(\mb{x})$. For QC-OPF, $\mb{H}$ is linear in $(\beta,\pmb{\alpha})$ and independent of $\mb{x}$ while all other quantities are either linear or quadratic in $\mb{x}$. \edit{The same also holds for CDF-OPF.} By pre-computing the coefficients of those terms, we achieve significant speedups since the computation of the above quantities reduces to sparse matrix-vector multiplications. 
		For example, our implementation solves PGLIB \texttt{case10000} in 50 seconds without any constraint removal, compared to 67 seconds for the Julia-based IPOPT package \emph{PowerModels.jl}~\cite{POWERMODELS}, despite lacking Julia's JIT features.

		Furthermore, quadratic thermal limits (for example,~\eqref{eq:qc-flowlims} for QC-OPF) add significant computation overhead due to their contribution of linear terms to $\nabla \mb{C}(\mb{x})$ and constant terms to $\sum_{j=1}^p \alpha_j \nabla^2 \mb{C}_j(\mb{x})$. However, most of these constraints are not binding at the optimum, as shown in Figure~\ref{fig:binding}. Therefore, reducing the solve time of any C-OPF involves accurately detecting and removing non-binding thermal limits, which are part of the constraints $\tilde{\mb{g}}(\mb{x}) \preceq \pmb{\gamma}$ in our model.

		\subsection{CS Classifiers}
		
		Our design of $\nn$ involves \icnn\ with $K=5$ and layer sizes $(2N+2E,N+E,300,150,2N+2E)$, and ELU ($\alpha=1$) activations. \mgn\ is chosen with $K=2$, $\sigma_k$ is chosen as ELU ($\alpha=1$) and Softplus, and $\mb{V}^\top\mb{V}$ is chosen to have rank 200 for all cases. \gate\ is a 2-layer network with a hidden layer of size 100, and LeakyReLU activation. The Adam optimizer is used to train $\nn$ for 2000 steps with a batch size of 64.
		Once $\nn$ has been trained via Algorithm~\ref{alg:dataset}, 100 randomly selected examples from the test set are fed in Algorithm~\ref{alg:accelerate} to observe the acceleration of C-OPF. We train $\nn$ via PyTorch, which is carried out on a system with an Intel Core i9 processor, 64 GB of RAM, and two NVIDIA RTX 3090 GPUs.

		We can compare the proposed $\nn$-based approach with other benchmark classification methods:
		\begin{itemize}
			\item To conduct an ablation study of \nn, we list the performance of standalone \icnn~and~\mgn. The former is a state-of-the-art architecture for learning dual variables in convex problems; see~\cite{YC-LZ-BZ-2022}. 
			\item \emph{Deep classifier} is an DNN, which takes $(\pmb{\gamma},\pmb{\xi})$ as input and produces $\tL$ scalars in the range $(0,1)$ as output, representing the probability of each of the $\tL$ constraints binding \cite{DD-SM-2019}. We employ a 4-layer DNN with same layer sizes as the $\nn$, and train it on dataset $\mcal{D}$ using binary cross-entropy loss. Training is performed for 2,000 epochs using the Adam optimizer with a batch size of 64.
			\item \emph{Ridge classifier} uses PyTorch to train a ridge regression classifier on $\mcal{D}$. It first converts the targets (indicating binding and non-binding constraints) to $\{-1, 1\}$, then fits a linear model that maps the input parameters to these targets using ridge regression.
			\item \emph{XGBoost classifier} involves training a gradient-boosted decision-tree classifier. The \texttt{XGBoost} package is used with a tree depth of 4, trained for 50 epochs at a learning rate of 0.1.
		\end{itemize}
		
		\subsection{Numerical Results}
		All classifiers' predictions are fed into Algorithm~\ref{alg:accelerate}. The comparative results are presented in Tables~\ref{tab:presolve_performance} and~\ref{tab:presolve_performance-cdf}, including the training time, confusion matrix, total solve times, and the fraction of cases requiring resolution due to the detection of violated constraints.
		
		Deep and ridge classifiers, along with \mgn, train the fastest, whereas $\nn$ takes longer due to its MoE-based framework. However, $\nn$'s training time scales well with the number of buses, unlike XGBoost, which becomes impractical for larger cases. For example, XGBoost training for \texttt{case10000} exceeds 12 hours and fails to complete.

		The proposed $\nn$-based CS method shows high true positives and low false negatives but has relatively lower true negatives and higher false positives. This outcome is acceptable because, in constraint classification, minimizing false negatives is critical to avoid constraint violations in the reduced problem, which would require re-solving and significantly increase the total runtime of Algorithm~\ref{alg:accelerate}. $\nn$ achieves very low false negatives, helping to avoid re-solves and resulting in the lowest solve times in most cases. Additionally, \moge\ achieves lower false negatives than \icnn\ and \mgn, which is beneficial for larger cases. Interestingly, \mgn\ tends to classify most constraints as binding, allowing \gate\ to switch to \mgn\ if \icnn\ makes a false positive error.
		
		The solve time advantage of \moge\ ranges from 10\% to 34\%, which is significant when C-OPF needs to be solved repeatedly or when computational resources are limited.Furthermore, \moge\ is more robust than \icnn\ in avoiding re-solves. Additionally, although the deep classifier and XGBoost offer good classification performance, they cannot match the efficiency of architectures like \moge\ and \icnn\ designed to approximate convex functions and their gradients.

		\section{Conclusion}
		\label{sec:5}
            \edit{
		In this paper, we introduced a data-driven constraint screening approach to called \moge\ to presolve convexified optimal power flow, utilizing neural network architecture that can learn convex gradients. The following analyses and experiments were carried out.
        \begin{itemize}
            \item We presented a general technique based on rank conditions to assess the suitability of a family of C-OPF models for constraint screening.
            \item We provided generalization results and a guarantee that our proposed framework leads to the exact solution as the original problem in finite time.
            \item We discussed methods for generating a dataset of C-OPF parameters, and augmenting the same.
            \item  Numerical simulations illustrate that the proposed method significantly reduces the solve times for large-scale C-OPF problems.
        \end{itemize}
        Future research directions include the following.
        \begin{itemize}
            \item Extending \moge\ to cases where the dual variables may be degenerate.
            \item Developing presolving techniques for power networks with changing topologies and discrete elements, e.g. switches.
        \end{itemize}
	} 

    \edit{
    \appendix
    \section{CDF-OPF}
    \label{app:cdf-opf}
        CDF-OPF is a convex relaxation of ACOPF for low-voltage radial networks with a tree topology. Let $(\mcal{N}_0,\mcal{E})$ be the directed graph representing the power network, where $\mcal{N}_0\ldef \{0\} \cup \mcal{N}$ is the set of buses with 0 denoting a special leaf bus (often the point of common coupling). We assume that the branches in $\mcal{E}$, denoted by $(i,j)$ are oriented \emph{away} from bus 0. Owing to the network's tree topology, each branch in $\mcal{E}$ corresponds to a unique bus in $\mcal{N}$, and as such we index a branch as $j$ rather than its full description $(i,j)$. Each bus $i\in\mcal{N}$ is associated with indices $\mcal{G}_i$ of generators attached to it, real and reactive demands $P_i^d,Q_i^d$, and minimum (maximum) voltage magnitude $\underline{V}_i$ ($\bar{V}_i$). We assume $V_0$ is fixed since it is determined by the upstream high-voltage transmission network. In addition, the real and reactive power outputs of the $\tth{k}$ generator are denoted by $P_k^g$ and $Q_k^g$ respectively. The set $\mcal{G} \ldef \cup_{i\in\mcal{N}} \mcal{G}_i$ collects all generation units. Each branch $i$ is represented by its resistance $r_i$, reactance $x_i$, and thermal limit $\bar{S}_i$. \par
        Let $V_i$ denote the squared voltage magnitude and $p_i + \mathsf{j}q_i$ the complex injection at bus $i$. Furthermore, let $P_i + \mathsf{j}Q_i$ denote the complex power flow and $\ell_i$ the squared current magnitude in the direction of the branch $i$. Lastly, we use $\pi_i$ to denote the \emph{parent} of bus $i$, i.e. $(\pi_i,i)\in\mcal{E}$, and the notation $i\rightarrow j$ to denote $(i,j)\in\mcal{E}$. In this setup, the CDF-OPF formulation is given as follows.
    \begin{subequations}
    \label{eq:cdf-opf}
        \begin{align}
            \label{eq:cdf-obj}
			&\min \, f\left( \left\{ P^g_k \right\}_{k\in\mcal{G}}, \left\{Q^g_k \right\}_{k\in\mcal{G}}, \left\{ V_i \right\}_{i\in\mcal{N}}, \left\{ \ell_i \right\}_{i\in\mcal{N}} \right),  \\ 
			&\text{subject to:} \notag \\	
            \label{eq:cdf-real-balance}
            & \sum_{k:i\rightarrow k} P_k - p_i - P_i + r_i\ell_i=0, \forall i\in\mcal{N}\\
            \label{eq:cdf-reac-balance}
            &\ \sum_{k:i\rightarrow k} Q_k - q_i - Q_i + x_i\ell_i=0, \forall i\in\mcal{N}\\
            \label{eq:cdf-volt-curr-power}
            & V_{\pi_i} - V_i - 2r_iP_i - 2x_iQ_i + (r_i^2+x_i^2)\ell_i = 0, \forall i \in \mcal{N}\\
            \label{eq:cdf-socp}
            & P_i^2 + Q_i^2 \leq V_{\pi_i}\ell_i, \forall i\in\mcal{N}\\
            \label{eq:cdf-real-inj}
            & p_i - \left( \sum_{k\in\mcal{G}_i} P^g_k \right) = -P^d_i, \forall i\in\mcal{N}\\
            \label{eq:cdf-reac-inj}
            & q_i - \left( \sum_{k\in\mcal{G}_i} Q^g_k \right) = -Q^d_i, \forall i\in\mcal{N}\\
            \label{eq:cdf-slack-bounds}
            & \underline{p}_0 \leq \sum_{k:0\rightarrow k} P_k \leq \bar{p}_0,\quad \underline{q}_0 \leq \sum_{k:0\rightarrow k} Q_k \leq \bar{q}_0\\
            \label{eq:cdf-flows}
            & P_i^2 + Q_i^2 \leq \bar{S}_i^2, \forall i \in \mcal{N}\\
            \label{eq:cdf-volt-balance}
            &\underline{V}_i^2 \leq V_i \leq \bar{V}_i^2, \forall i\in\mcal{N}\\
            \label{eq:cdf-gen-limits}
            & \underline{P}^g_k \leq P^g_k \leq \bar{P}^g_k, \quad \underline{Q}^g_k \leq Q^g_k \leq \bar{Q}^g_k, \forall k \in \mcal{G}
        \end{align}
    \end{subequations}
    \begin{table*}[tb!]
        \centering
        \caption{Representing CDF-OPF~\eqref{eq:cdf-opf} as parametric convex optimization problem~\eqref{prob:orig}, where the box constraint~\eqref{eq:var-bounds} corresponds to constraints~\eqref{eq:cdf-volt-balance} and \eqref{eq:cdf-gen-limits}.}
        \resizebox{\linewidth}{!}{
        \renewcommand{\arraystretch}{1.4}
\begin{tabular}{l p{3cm}|l p{4cm}|l p{4cm}}
	\hline
	\multicolumn{2}{c|}{\textbf{Variables}} & \multicolumn{2}{c|}{\textbf{Constraints}} & \multicolumn{2}{c}{\textbf{Parameters}}\\ \hline
	\multirow{3}{*}{$\mb{x}$:} & \multirow{3}{*}{\begin{minipage}{0.2\linewidth}
		$\{P^g_k,Q^g_k\}_{k\in\mcal{G}}$,\\$\{P_{i},Q_{i},\ell_i \}_{i\in\mcal{N}}$,\\
			$\{ p_i, q_i, v_i \}_{i\in\mcal{N}}$
	\end{minipage}} & $g(\mb{x}), \tilde{g}(\mb{x})$: & \eqref{eq:cdf-socp},~\eqref{eq:cdf-slack-bounds}--\eqref{eq:cdf-flows} &$\pmb{\gamma}$: & \multirow{2}{*}{\begin{minipage}{0.2\linewidth}
	$\underline{p}_0,\bar{p}_0,\underline{q}_0,\bar{q}_0,\{\bar{S}_{i}^2 \}_{i\in\mcal{N}}$\\{}
	\end{minipage}}\\
	&  &$h(\mb{x}), \tilde{h}(\mb{x})$: & \eqref{eq:cdf-real-balance}--\eqref{eq:cdf-volt-curr-power},~\eqref{eq:cdf-real-inj}--\eqref{eq:cdf-reac-inj}  & $\pmb{\xi}$: & $\{P^d_i,Q^d_i\}_{i\in\mcal{N}}$\\
	\hline
\end{tabular}
\renewcommand{\arraystretch}{1}
        }
        \label{tab:convert_to_param-cdf}
    \end{table*}
    In the above, the convex objective~\eqref{eq:cdf-obj} can model various objectives such as generation cost minimization, volt-VAR control, or minimization of $I^2R$ losses in the network. Constraints~\eqref{eq:cdf-real-balance}--\eqref{eq:cdf-socp} represent the convex relaxation of the DistFlow equations.~\eqref{eq:cdf-real-inj}--\eqref{eq:cdf-reac-inj} characterize per-bus injections, while~\eqref{eq:cdf-slack-bounds} bound the injections on bus 0. The thermal limits of each branch are specified in~\eqref{eq:cdf-flows}. Lastly,~\eqref{eq:cdf-volt-balance} places bounds on voltage magnitudes, while~\eqref{eq:cdf-gen-limits} limits real and reactive generation quantities. CDF-OPF can be represented as a parameterized C-OPF~\eqref{prob:orig} as shown in Table~\ref{tab:convert_to_param-cdf}.
    \begin{lemma}
    \label{lem:cdf-fullrank}
        If $V_0>0$ and the feasible set of~\eqref{eq:cdf-reac-balance}--\eqref{eq:cdf-socp} does not include $V_i=0$ for any bus $i$, Lemma~\ref{lemma:fixLICQ} holds for CDF-OPF.
    \end{lemma}
    \begin{proof}
    Let $N \ldef |\mcal{N}|$, and let $\tilde{\mb{A}}\in\real{N \times N+1}$ be the branch-bus incidence matrix of $(\mcal{N}_0,\mcal{E})$. We can represent $\tilde{\mb{A}}$ as $\tilde{\mb{A}} = \bm{\mb{a}_0 & \mb{A}}$ wherein $\mb{A}\in\real{N \times N}$ is the invertible reduced branch-bus indicence matrix. Subsequently, we set up a few definitions. We let $\mb{F}\ldef \mb{A}^{-1}$, $\mb{D}_r\ldef \text{diag}\left( \{ r_i\}_{i\in\mcal{N}} \right)$ and $\mb{D}_x\ldef \text{diag}\left( \{ x_i\}_{i\in\mcal{N}} \right)$, and correspondingly the matrices $\mb{R}\ldef 2\mb{F}\mb{D}_r\mb{F}^\top$ and $\mb{X}\ldef 2\mb{F}\mb{D}_x\mb{F}^\top$ are defined. In this setting, equations~\eqref{eq:cdf-real-balance}--\eqref{eq:cdf-socp} can be written in matrix-vector notation as~\cite[(7)]{dg-topologies}
    \begin{subequations}
        \begin{gather}
            \label{eq:cdf-1}
            \mb{p} - \mb{A}^\top \mb{P} - \mb{D}_r \pmb{\ell} = \mb{0}\\
            \label{eq:cdf-2}
            \mb{q} - \mb{A}^\top \mb{Q} - \mb{D}_x \pmb{\ell} = \mb{0}\\
            \label{eq:cdf-3}
            \mb{Av} + \mb{a}_0 V_0 - 2\mb{D}_r\mb{P} - 2\mb{D}_x\mb{Q} + (\mb{D}_r^2 + \mb{D}_x^2)\pmb{\ell} = \mb{0}\\
            \label{eq:cdf-4}
            \text{diag}(\mb{P})\mb{P} + \text{diag}(\mb{Q})\mb{Q} - \text{diag}(\pmb{\ell})\mb{v}_\pi \preceq \mb{0},
        \end{gather}
        wherein $\mb{p} \ldef \{p_i\}_{i\in\mcal{N}}$, $\mb{q} \ldef \{q_i\}_{i\in\mcal{N}}$, $\mb{P} \ldef \{P_i\}_{i\in\mcal{N}}$, $\mb{Q} \ldef \{Q_i\}_{i\in\mcal{N}}$, and $\mb{v}\ldef \{V_i\}_{i\in\mcal{N}}$ gather the relevant variables in a vectorized form. Furthermore $\mb{v}_\pi\ldef\bm{V_{\pi_1} & V_{\pi_2} & \cdots & V_{\pi_N}}$ collects the parent bus' voltage of each bus. It remains to prove that the Jacobian of binding constraints in~\eqref{eq:cdf-1}--\eqref{eq:cdf-4} are linearly independent with respect to any feasible $\mb{x}$ as defined in Table~\ref{tab:convert_to_param-cdf}.\par
        To that end, we show linear independence of Jacobian rows with respect to a subset of variables (columns) given as
        \begin{align*}
            \tilde{\mb{x}} \ldef \{ \mb{p}, \mb{q}, \mb{v}, [\pmb{\ell}]_{\mcal{A}_{\mb{x}}}\},
        \end{align*}
        wherein $\mcal{A}_{\mb{x}}\subseteq [N]$ is the set of binding SOCP constraints~\eqref{eq:cdf-4} for a given $\mb{x}$. Note that 
        \begin{align*}
            \frac{\partial \eqref{eq:cdf-1},\eqref{eq:cdf-2},\eqref{eq:cdf-3},\eqref{eq:cdf-4} }{\partial \tilde{\mb{x}} } = \text{blkdiag}\bm{\mb{I}_N& \mb{I}_N & \mb{A} & [\text{diag}(\mb{v}_{\pi})]_{\mcal{A}_{\mb{x}}}},
        \end{align*}
        where $[\text{diag}(\mb{v}_{\pi})]_{\mcal{A}_{\mb{x}}}$ selects the rows of $\text{diag}(\mb{v}_{\pi})$ corresponding to the indices in $\mcal{A}_{\mb{x}}$. By the invertibility of $\mb{A}$ and the nonzero nature of the vector $\mb{v}_\pi$ due to feasibility of $\mb{x}$, we conclude that the above partial jacobian has full row-rank. This concludes the proof.
    \end{subequations}
    \end{proof}
    Lemma~\ref{lem:cdf-fullrank} shows that the convexified DistFlow model, under mild conditions (i.e. the voltage collapse mode $V_i=0$ for any bus $i$ is not feasible) satisfies the fixed LICQ conditions. Subsequently, it also satisfies the strong duality conditions over $\Gamma \times \Xi$ stated in Theorem~\ref{thm:LICQ}.
    }
	
	\section{Proof of Lemma~\ref{lemma:fixLICQ}}\label{proof:fixLICQ}
	Let 
	$\mcal{A}_{\mb{g}}(\mb{x})  = \mcal{A}_\soc(\mb{x}) \cup \mcal{A}_\ang(\mb{x})$,
	where 
	$\mcal{A}_\soc(\mb{x})$ is the set of SOC binding constraints in~\eqref{eq:qc-soc}. 
	$\mcal{A}_\ang(\mb{x}) \ldef \mcal{A}^{low}_\ang(\mb{x}) \cup \mcal{A}^{up}_\ang(\mb{x})$ is the set of angle binding constraints in~\eqref{eq:qc-ang-lims}, which reflects the bindings at the PAD lower or upper limits.\footnote{Henceforth, we omit the argument $(\mb{x})$ for these sets to simplify notation.} 
	Also, Let $l(m):[E]\mapsto\mcal{E}$ be the function mapping each index in $[E]$ to a unique line in $\mcal{E}$.
	Define the partial set of variables $\mb{x}_s\ldef \left\{P_{ij}, Q_{ij}, P_{ji}, Q_{ji}, W^{\text{R}}_{ij}, W^{\text{I}}_{ij}\right\}_{(i,j)\in\mcal{E}}$. Let $\mb{J}_s$ be the submatrix of the Jacobian $\mb{J}$ formed by the columns corresponding to $\mb{x}_s$. We can compactly express it as a block upper triangular matrix of size
	$(4E + |\mcal{A}_{\mb{g}}(\mb{x})|)\times 6E$, as follows:
	\begin{align*}
		\mb{J}_s = 
		\begin{NiceMatrix}
			& \scriptstyle{P_{ij}} & \scriptstyle{Q_{ij}} & \scriptstyle{P_{ji}} & \scriptstyle{Q_{ji}} & \scriptstyle{W^{\text{R}}_{ij}} & \scriptstyle{W^{\text{I}}_{ij}}\\
			& \Block{3-4}{\mb{I}_{4E}} &&&& \Block{3-2}{\boldsymbol{\ast}} \\
			\scriptstyle{\eqref{eq:qc-flows}}     & \\
			& \\
			\scriptstyle{\eqref{eq:qc-soc}}      & \Block{2-4}{\mb{0}} &&&& \Block{2-2}<\footnotesize>{\\ \\ \underbrace{\bm{\mb{S}^{\text{R}} & \mb{S}^{\text{I}}\\\mb{A}^{\text{R}} & \mb{A}^{\text{I}}}}_{\mb{J}_g}} \\
			\scriptstyle{\eqref{eq:qc-ang-lims}}  &  & 
			\CodeAfter
			\SubMatrix({2-2}{6-7})[vlines=4,hlines=3]
		\end{NiceMatrix}
	\end{align*}
	where the four submatrices are given as:	\begin{align*}
		\mb{S}^{\text{R}} &= \left[  2W^{\text{R}}_{l(m)}\mb{e}_{m}^\top:m\in\mcal{A}_\soc \right] \\
		\mb{S}^{\text{I}} &= \left[  2W^{\text{I}}_{l(m)}\mb{e}_{m}^\top:m\in\mcal{A}_\soc  \right]\\
		\mb{A}^{\text{R}} &= 
		\left[ 
		\begin{cases}
			\tan(\underline{\theta}_{l(m)})\mb{e}_m^\top& \text{if }m\in\mcal{A}_\ang^{low}\\
			-\tan(\bar{\theta}_{l(m)})\mb{e}_m^\top& \text{if }m\in\mcal{A}_\ang^{up} 
		\end{cases}
		\right]\\
		\mb{A}^{\text{I}} &= 
		\left[ 
		\begin{cases}
			-\mb{e}_m^\top& \text{if }m\in\mcal{A}_\ang^{low} \\
			\mb{e}_m^\top& \text{if }m\in\mcal{A}_\ang^{up} 
		\end{cases}
		\right].
	\end{align*}

	\noindent Note that these submatrices are constructed row by row using scaled canonical vectors $\{\mb{e}_m^{\top}\}$, with scaling factors determined by specific conditions.
	Furthermore, we have $W^{\text{R}}_{l(m)}\neq 0, \forall m\in\mcal{A}_\soc(\mb{x})$. Otherwise, $W^{\text{I}}_{l(m)} = 0$ due to \eqref{eq:qc-ang-lims}, which implies $W_{ii} = 0$ for some $i$ (cf. binding constraint \eqref{eq:qc-soc}). This would violate the box constraint~\eqref{eq:qc-volt-lims} with nonzero $\underline{V}_i^2$.

	We will explore the special structure of $\mb{J}_g(\mb{x})$ to demonstrate its full rank for any $\mb{x} \in \mcal{X}_{\mathrm{pf}}$, thereby establishing the full rank of $\mb{J}_s$ and $\mb{J}$~\cite[Sec. 0.9.4]{RAH-CRJ:2012}. First, if $\mcal{A}_{\mb{g}} =\varnothing$, then $\mb{J}_s = \bm{\mb{I}_{4E} & \boldsymbol{\ast}}$, which is full rank.
	Otherwise, depending on the binding possibilities of~\eqref{eq:qc-soc} and~\eqref{eq:qc-ang-lims}, we have four cases: 
	i) $\mcal{A}_\soc=\varnothing, \mcal{A}_\ang \neq \varnothing$; ii)
	$\mcal{A}_\soc\neq \varnothing, \mcal{A}_\ang =\varnothing$; iii) $\mcal{A}_\soc \neq \varnothing, \mcal{A}_\ang \neq \varnothing, \mcal{A}_\soc \cap \mcal{A}_\ang = \varnothing$; and iv) $\mcal{A}_\soc \cap \mcal{A}_\ang \neq \varnothing$.
	In each of the first three cases, $\mathbf{J}_g$ is full rank as it is in \emph{row echelon form (REF) without zero rows}, possibly after row permutations.
	
	Finally, the most complex case-iv occurs when both the SoC constraint and one of the angle limit constraints are simultaneously binding for at least one line. 
	Define $\rho_m \ldef \frac{-\tan(\underline{\theta}_{l(m)})}{2W^{\text{R}}_{l(m)}},\,\forall m \in \mcal{A}_\soc \cap \mcal{A}_\ang^{low}$. For each $m$, the two corresponding rows in $\mb{J}_g$ have only four nonzero elements. Applying the elementary row operation:
	\begin{align*}
		&\begin{bNiceMatrix}
			\cdots &2W^{\text{R}}_{l(m)} &\cdots &2W^{\text{I}}_{l(m)}&\cdots\\
			\cdots &\tan(\underline{\theta}_{l(m)}) &\cdots  &-1 &\cdots
		\end{bNiceMatrix}
		\ro{r_2+\rho_m \times r_1}  \\
		&\begin{bNiceMatrix}
			\cdots  & \hspace{0.5cm} 2W^{\text{R}}_{l(m)} & \hspace{0.23cm} \cdots &2W^{\text{I}}_{l(m)} &\cdots\\
			\cdots &0 &\hspace{0.23cm} \cdots  &p &\cdots
		\end{bNiceMatrix}\,,
	\end{align*}
	we get the pivot
	$p=-(1+\tan^2(\underline{\theta}_{l(m)}))\neq 0$ for the second row. Similarly, define $\tilde{\rho}_m \ldef \frac{\tan(\bar{\theta}_{l(m)})}{2W^{\text{R}}_{l(m)}},\,\forall m \in \mcal{A}_\soc \cap \mcal{A}_\ang^{up}$, we get the pivot $\tilde{p}=1+\tan^2(\bar{\theta}_{l(m)})\neq 0$ with the same row operation.
	By repeating the aforementioned operation for all $m \in \mcal{A}_\soc \cap \mcal{A}_\ang$, we convert $\mb{J}_g$ into \emph{REF without zero rows}, confirming its full rank.
	
	\section{Proof of Theorem~\ref{thm:LICQ}}
	\label{app:proof-thm-LICQ}
	Consider the following parametric (not necessarily convex) problem, where the parameter space $\Theta$ is an open, nonempty set with non-zero measure.
	\begin{align*}
		\mcal{P}(\pmb{\theta}) \ldef \min\limits_{\mb{x}\in\real{n}}\quad& f(\mb{x})\\
		\text{s.t.}\quad& \mb{g}(\mb{x}) \mleq \mb{0},\quad \mb{h}(\mb{x})=\mb{0}\\
		\quad& \tilde{\mb{g}}(\mb{x},\pmb{\theta})\mleq \mb{0},\quad \tilde{\mb{h}}(\mb{x},\pmb{\theta}) = \mb{0}.
	\end{align*}
	\begin{thm*}[Genericity of LICQ~\cite{AH-SB-GH-FD:2018}]
		Consider the problem $\mcal{P}(\pmb{\theta})$ with $f,\mb{g},\mb{h},\tilde{\mb{g}},\tilde{\mb{h}}$ in class $C^r$ (i.e., $r$-times continuously differentiable) for all $\mb{x}\in\mcal{X}_{\mathrm{pf}}$ and 
		$\pmb{\theta}\in\Theta$ with $r>\max(0,n-M-\tilde{M})$ and
		assume LICQ holds for all $\mb{x}\in\mcal{X}_{\mathrm{pf}}$.
		If the map
		$
		\pmb{\theta} \mapsto (\tilde{\mb{g}}(\mb{x},\pmb{\theta}),\tilde{\mb{h}}(\mb{x},\pmb{\theta}))
		$
		has rank $\tilde{M}+\tilde{L}$ for every $\mb{x}\in\mcal{X}_{\mathrm{pf}}$ and $\pmb{\theta}\in\Theta$,
		then LICQ holds for almost every $\pmb{\theta}\in\Theta$ and for every feasible point of $\mcal{P}(\pmb{\theta})$.
	\end{thm*}
	To prove Theorem~\ref{thm:LICQ}, we will demonstrate that QC-OPF satisfies all conditions stated in the theorem above. That is, C1) all functions are in class $C^r$; C2) the parameter-to-constraint map is full rank; and C3) LICQ holds for all $\mb{x}\in\mcal{X}_{\mathrm{pf}}$, which is already established in Lemma~\ref{lemma:fixLICQ}.
	First, since all functions in QC-OPF are either affine or quadratic, they are infinitely differentiable, thereby satisfying condition C1). Second, let us consider embedding the box constraint $\underline{\mb{x}} \preceq \mb{x} \preceq \bar{\mb{x}}$ into $\tilde{\mb{g}}$.
	The parameter-to-constraint map of C-OPF is $(\pmb{\gamma},\pmb{\xi}) \mapsto (\tilde{\mb{g}}(\mb{x})-\pmb{\gamma},\tilde{\mb{h}}(\mb{x})-\pmb{\xi})$.
	It is evident that the mapping's Jacobian is the diagonal sign matrix ($\pm 1$ on the diagonal). Thus, the map 
	is full-rank. Finally, it is known that LICQ implies \emph{Mangasarian-Fromovitz constraint qualification (MFCQ)}, which is equivalent to Slater's condition for convex problems~\cite[pp. 45, 160]{JB-AL:2006}.
	This result ensures strong duality and hence completes the proof. 
	

	\section{Proof of Proposition~\ref{prop:two-iters}}\label{app:proof_two_iters}
    \label{app:maxloops}
	For brevity, we define the set
	\begin{align*}
		C(\pmb{\xi}) \ldef \left\{ \mb{x}\in\real{n}: \mb{g}(\mb{x})\preceq \mb{0}, \mb{h}(\mb{x})=\mb{0},\tilde{\mb{h}}(\mb{x}) = \pmb{\xi}, \underline{\mb{x}} \preceq \mb{x} \preceq \bar{\mb{x}}\right\},
	\end{align*}
	which allows us to restate~\eqref{prob:orig} as
	\begin{align}
		\label{eq:restate}
		\mcal{V}(\pmb{\gamma},\pmb{\xi}) = \min\limits_{\mb{x}\in C(\pmb{\xi})} \quad& f(\mb{x}),\quad \text{s.t.} \quad \tilde{\mb{g}}(\mb{x})\preceq \pmb{\gamma},
	\end{align}
	with its optimizer denoted by $\mb{x}^*$. Consider the following lemma.
	\begin{lemma}
		\label{claim:onecons}
		For problem~\ref{eq:restate}, suppose it is known \emph{a priori} that $\tilde{\pmb{\lambda}}^*_i > 0$ for $i\in S\subseteq [\tilde{L}]$. Then, the optimizer $\mb{x}^{\epsilon}$ of the perturbed problem $\mcal{V}(\pmb{\gamma}+\epsilon\Delta,\pmb{\xi})$, where
		\begin{align*}
			\Delta \ldef \begin{cases}
				1, & \text{if } i \in S\\
				0, & \text{otherwise}
			\end{cases}
		\end{align*}
		violates at least one constraint in $[\tilde{\mb{g}}(\mb{x})]_S \preceq [\pmb{\gamma}]_S$, for any $\epsilon>0$.
	\end{lemma}
	\begin{proof}
		Since it is known \emph{a priori} that $[\tilde{\pmb{\lambda}}^*]_S = -[\nabla_{\pmb{\gamma}} \mcal{V}(\pmb{\gamma},\pmb{\xi})]_S \succ \mb{0}$, and from the fact that $\nabla_{\pmb{\gamma}}\mcal{V}(\pmb{\gamma},\pmb{\xi})\preceq \mb{0}$ for any feasible $(\pmb{\gamma},\pmb{\xi})$ due to Prop.~\ref{prop:grad} and dual variable properties, it follows that $\mcal{V}(\pmb{\gamma},\pmb{\xi})>\mcal{V}(\pmb{\gamma}+\epsilon\Delta,\pmb{\xi})$ for all $\epsilon > 0$. Consequently, we have $f(\mb{x}^*)>f(\mb{x}^\epsilon)$. To prove by contradiction, assume that $[\tilde{\mb{g}}(\mb{x}^\epsilon)]_S \preceq [\pmb{\gamma}]_S$, i.e. the perturbed solution $\mb{x}^\epsilon$ does not violate any constraints in~\eqref{eq:restate}. However, this implies that $\mb{x}^\epsilon$ is the optimizer to~\eqref{eq:restate} rather than $\mb{x}^*$. This leads to a contradiction, showing that the assumption is false.
	\end{proof}
	We set $S$ to be the set of false negative constraints. Lemma~\ref{claim:onecons} guarantees that at least one constraint violation will be detected in every iteration of the do-while loop and added to $\mcal{A}_{\text{bind}}$. Thus, if there are $p$ false negatives initially, the do-while loop must run no more than $p+1$ iterations.

    \edit{
    \section{Proof of Theorem~\ref{thm:moge-gen}}
    \label{app:moge-gen-proof}
    We start this proof by observing certain properties of convex functions and their gradients. Consider a monotone function $f:\real{n}\mapsto \real{n}$ such that $f = \nabla_{\mb{x}}\Phi(\mb{x})$ for some convex $\Phi:\mcal{X}\mapsto\real{}$ over $\mcal{X}$. Let the discrete set $\bar{\mcal{X}}\ldef\{\mb{x}^{(1)},\cdots,\mb{x}^{(K)}\}\subset \mcal{X}$ be such that $\mb{x}^{(i)}\in\text{dom}(\Phi)$ and $f(\mb{x}^{(i)})=\mb{g}$ for all $i\in[K]$. Let an arbitrary point $\tilde{\mb{x}}\in\text{conv}\left(\bar{\mcal{X}}\right)$ be represented as $\tilde{\mb{x}} = \sum_{i\in[K]}\pmb{\alpha}_i \mb{x}^{(i)}$, where $\pmb{\alpha}\succeq \mb{0}$ an $\mb{1}^\top \pmb{\alpha}=1$. It follows from the convexity of $\Phi$ that
    \begin{align}
    \label{eq:convx-property}
        \Phi(\tilde{\mb{x}}) \geq \Phi(\mb{x}^{(i)}) + \left(\nabla \Phi(\mb{x}^{(i)})\right)^\top \left(\tilde{\mb{x}} - \mb{x}^{(i)} \right).
    \end{align}
    Multiplying~\eqref{eq:convx-property} with $\mb{\alpha}_i$ and adding the inequalities over all $i$, we get
    \begin{align}
    \label{eq-thm2-eq1}
    \Phi \left( \sum_{i\in[K]} \pmb{\alpha}_i\mb{x}^{(i)} \right) \geq
        \sum_{i\in[K]} \pmb{\alpha}_i \Phi(\mb{x}^{(i)}).
    \end{align}
    On the other hand, Jensen's inequality provides
    \begin{align}
    \label{eq-thm2-eq2}
        \Phi \left( \sum_{i\in[K]} \pmb{\alpha}_i\mb{x}^{(i)} \right) \leq
        \sum_{i\in[K]} \pmb{\alpha}_i \Phi(\mb{x}^{(i)}).
    \end{align}
    \eqref{eq-thm2-eq1} and~\eqref{eq-thm2-eq2} imply that equality holds both for any value of $\pmb{\alpha}$ in the standard simplex. Taking the gradient of both sides of this equality with respect to $\mb{x}$, we obtain $f \left( \mb{x} \right) = \sum_{i\in[K]} \pmb{\alpha}_i f(\mb{x}^{(i)}) = \mb{g}$ for all $\mb{x}\in\text{conv}\left(\bar{\mcal{X}}\right)$.\par
    In the context of the the present Theorem, we replace $\mcal{X}$ with $\mcal{A}$ and $f$ with $\nabla \mcal{V}$, $\mcal{G}$, and $\mcal{H}$ to observe that all three are identically equal to $(-\tilde{\pmb{\lambda}}^*,-\tilde{\pmb{\mu}}^*)$ over $\text{conv}(\mcal{A})$. Correspondingly the convex combination $\mcal{Z}$ is also identically equal to $(-\tilde{\pmb{\lambda}}^*,-\tilde{\pmb{\mu}}^*)$ over $\text{conv}(\mcal{A})$.
    }

\end{document}